\documentclass[10pt,final]{amsart}
\usepackage{amsaddr}
\usepackage{mathpazo}

\usepackage[centering]{geometry}
\usepackage{enumerate}
\usepackage{amsmath}
\usepackage{amsthm}
\usepackage{amssymb}
\usepackage{amscd}
\usepackage{amsfonts}
\usepackage{amsbsy}
\usepackage{graphicx}
\usepackage{caption}
\usepackage{subcaption}
\usepackage[dvips]{psfrag}
\usepackage{array}
\usepackage{color}
\usepackage{epsfig}
\usepackage{url}
\usepackage{overpic}
\usepackage{epstopdf}
\usepackage[normalem]{ulem}
\usepackage{tikz}
\usepackage{overpic}
\usepackage{multirow}
\usetikzlibrary{arrows}
\usepackage{overpic}
\usepackage{float}
\usepackage{hyperref}
\usepackage{diagbox}

\newcommand{\R}{\ensuremath{\mathbb{R}}}

\newcommand{\CC}{\mathcal{C}}
\newcommand{\CS}{\ensuremath{\mathcal{S}}}

\newcommand{\CF}{\ensuremath{\mathcal{F}}}

\newcommand{\CO}{\ensuremath{\mathcal{O}}}
\newcommand{\CZ}{\ensuremath{\mathcal{Z}}}

\newcommand{\ov}{\overline}
\newcommand{\la}{\lambda}

\newcommand{\A}{\ensuremath{\mathcal{A}}}

\newcommand{\X}{\ensuremath{\mathcal{X}}}

\newcommand{\x}{\mathbf{x}}
\newcommand{\y}{\mathbf{y}}

\newcommand{\sgn}{\mathrm{sign}}

\newcommand{\de}{\delta}

\renewcommand{\d}{\mathrm{d}} 
 
\def\p{\partial}
\def\e{\varepsilon}

\newtheorem {theorem} {Theorem}
\newtheorem {definition} {Definition}
\newtheorem {proposition} {Proposition}
\newtheorem {corollary} {Corollary}
\newtheorem {lemma}  {Lemma}

\newtheorem {remark} {Remark}

\allowdisplaybreaks
\begin{document}
\renewcommand{\arraystretch}{1.5}

\title[Poincar\'e-Hopf Theorem for Filippov vector fields]
{Poincar\'e-Hopf Theorem for Filippov vector fields\\ on 2-dimensional compact manifolds}
\author[J.A. Casimiro, R.M. Martins, D.D. Novaes]
{Joyce A. Casimiro, Ricardo M. Martins$^{\ast}$, and Douglas D. Novaes}

\address{Departamento de Matem\'{a}tica, Instituto de Matem\'{a}tica, Estat\'{i}stica e Computa\c{c}\~{a}o Cient\'{i}fica (IMECC), Universidade
Estadual de Campinas (UNICAMP), Rua S\'{e}rgio Buarque de Holanda, 651, Cidade Universit\'{a}ria Zeferino Vaz, 13083--859, Campinas, SP,
Brazil.} \email{JAC: jcasimiro@ime.unicamp.br, DDN: ddnovaes@unicamp.br, RMM: rmiranda@unicamp.br.}
\thanks{$^\ast$Corresponding author.}



\begin{abstract}
The Poincaré–Hopf Theorem relates the Euler characteristic of a 2-dimensional compact manifold to the local behavior of smooth vector fields defined on it. However, despite the importance of Filippov vector fields, concerning both their theoretical and applied aspects, until now, it was not known whether this theorem extends to Filippov vector fields. In this paper, we demonstrate that the Poincaré–Hopf Theorem applies to Filippov vector fields defined on 2-dimensional compact manifolds with smooth switching manifolds. As a result, we establish a variant of the Hairy Ball Theorem, asserting that ``any Filippov vector field on a sphere with smooth switching manifolds must have at least one singularity (in the Filippov sense) with positive index''. This extension is achieved by introducing a new index definition that includes the singularities of Filippov vector fields, such as pseudo-equilibria and tangential singularities. Our work extends the classical index definition for singularities of smooth vector fields to encompass those of Filippov vector fields with smooth switching manifolds. This extension is based on an invariance property under a regularization process, allowing us to establish all classical index properties. We also compute the indices of all generic $\Sigma$-singularities and some codimension-1 $\Sigma$-singularities, including fold-fold tangential singularities, regular-cusp tangential singularities, and saddle-node pseudo-equilibria.
\end{abstract}

\maketitle


\section{Introduction}
The Poincaré--Hopf Theorem is a classical result that relates the Euler characteristic of a compact manifold with the indices of the singularities of smooth vector fields defined on it (see, for instance, \cite{fulton1995,MR0226651}). A well known application of such a theorem is the {\it Hairy Ball Theorem} which asserts that any smooth vector field defined on a sphere has at least one singularity (see, for instance, \cite{10.2307/2320860}). This result, although it may seem purely theoretical, is also useful in applied areas of science (see, for instance, \cite{10.2307/25151781,9372851}).

On the other hand, Filippov vector fields constitute an important class of dynamical systems, mainly because of their wide range of applications in many areas of science (see, for instance, \cite{MR4297797,MR2994324}). Roughly speaking, Filippov vector fields are piecewise smooth vector fields for which the local trajectories at points of non-smoothness are provided by the Filippov's convention. The concept of singularity for Filippov vector fields encompasses the usual one (for smooth vector fields), but also comprehend some new kinds of points over the non-smoothness set, namely, pseudo-equilibria and tangency points (see, for instance \cite{guardia2011generic}). The formal definition of Filippov vector fields and their singularities will be provided in Section \ref{sec:FVF}.

So far, a version of the Poincaré--Hopf Theorem for Filippov vector fields is not known. This is mainly because of the lack of a nice index definition for singularities in this context. As expected, a version of the Hairy Ball Theorem for Filippov vector fields is not known either. In other words, the  following  question is  open: ``Is there any Filippov vector field  defined on a sphere without singularities?'' 

In this paper, we are firstly concerned in extending the classical index definition to singularities of Filippov vector fields. Such an extension is provided by Definitions \ref{indexfilippov}, \ref{indexp}, and \ref{indiceM}, which are based on an invariance property under a regularization process established by Theorem \ref{thm:invreg}. With this new index definition, we are able to state and prove our main result, the Poincaré--Hopf Theorem for Filippov vector fields defined on $2$-dimensional compact manifolds with smooth switching manifolds (Theorem \ref{teoremapoincare}). Consequently, we also get a Hairy Ball Theorem in this context, i.e. ``any Filippov vector field defined on a sphere must have at least one singularity (in the Filippov sense) with positive index''.

This paper is structured  as  follows.  Section \ref{sec:FVF} is devoted to discuss the basic notions and definition of Filippov vector fields.  The definition of index  for singularities of Filippov vector fields is provided in Section \ref{sec:def} and some of their properties are established in Section \ref{sec:prop}.  Sections \ref{sec:cod1} and \ref{sec:generic} are devoted to computing the indices of some singularities of Filippov vector fields. Our main result, the Poincaré--Hopf Theorem for Filippov vector fields defined on 2-dimensional compact manifolds with smooth switching manifolds, is then stated and proven in Section \ref{sec:PHTthmFVF}. Section \ref{sectiom:invariaceunderreg} is dedicated to discuss the invariance property of this new index definition under a regularization process and to presenting a proof for Theorem \ref{thm:invreg}. An Appendix is also provided with some concepts and properties of the classical index for singularities smooth vector fields.

\section{Basic notions on Filippov vector fields}\label{sec:FVF}

In this section, we introduce the Filippov's convention for piecewise smooth vector fields defined on 2-dimensional compact manifolds. We also introduce the concept of singularities of Filippov vector fields. In what follows, smooth simply means $C^l$, $l\geq 1$.

First, let $M$ be a smooth $2$-dimensional compact manifold and $N\subset M$ be a smooth $1$-dimensional compact submanifold of $M$. Denote by $C_i$, $i\in\{1,2,\ldots,k\}$, the connected components of $M\setminus N$ (which is a finite amount because of the compactness of $M$ and $N$). Let $\X_i: M\to TM$, for $i\in\{1,2,\ldots,k\}$, be smooth vector fields on $M$, i.e. $\X_i(p)\in T_pM$ for every $p\in M$. Accordingly, we consider a piecewise smooth vector field on $M$ given by
\begin{equation}\label{dds}
\CZ(p)=\X_i(p)\,\,\textrm{if}\,\, p\in C_i,\,\,\text{for}\,\, i\in\{1,2,\ldots,k\},
\end{equation}
for which $N$ is called  {\it switching manifold}. It is worth mentioning that $N$ is not assumed to be connected, which allows $\CZ(p)$ to have several smooth subsystems.

The trajectories of \eqref{dds}, for points in $N,$ can be locally described by the Filippov's convention (see \cite{F}). To do so, we start by obtaining a description of \eqref{dds} in local coordinates around points of the switching manifold $N$. Since $N$ is a $1$-dimensional compact submanifold of $M$, we can find, for each $q\in N,$ a chart $(U,\Phi)$ of $M$ around $q$ (i.e. $\Phi:U\to\R^2$ is a smooth local coordinate system and $U\subset M$ is a neighborhood of $q$) and a smooth function $H:U\to\R$, having $0$ as a regular value, such that 
\begin{itemize}
\item $S:=N\cap U=H^{-1}(0)$, and
\item $U\setminus S$ is composed by two disjoint open sets, $S^+=\{p\in U: H(p)\geq0 \}$ and $S^-=\{p\in U:H(p)\leq 0 \},$ such that $\CZ^+=\CZ|_{S^+}$ and $\CZ^-=\CZ|_{S^-}$ are smooth vector fields.
\end{itemize}
Let $D=\Phi(U)$ and consider the following smooth vector fields defined on $D$
$$F^+:=\Phi_* \CZ^+:\Sigma^+\to\R^2\,\,\, \text{and}\,\,\, F^-:=\Phi_* \CZ^-:\Sigma^ -\to\R^2$$ 
(pushforward of $\CZ^+$ and $\CZ^-$ by $\Phi$,  respectively). In addition,  denote
\[
\Sigma^+:=\{\x\in D: f(\x)\geq0 \}=\Phi(S^+)\,\,\, \text{and}\,\,\,
\Sigma^-:=\{\x\in D: f(\x)\leq0 \}=\Phi(S^-).
\]
Acordingly, the piecewise smooth vector field \eqref{dds} can be locally described around $q\in N$ by the following piecewise smooth vector field on $D$ (see Figure \ref{localdesc}),
\begin{equation}\label{locdds}
Z(\x)=\Phi_*(\CZ |_{U}):=\left\{\begin{array}{l}
F^+(\x),\,\,\,\textrm{if}\,\,\, f(\x)\geq0,\vspace{0.1cm}\\
F^-(\x),\,\,\,\textrm{if}\,\,\, f(\x)\leq0.
\end{array}\right. 
\end{equation}
Usually, the Filippov vector field \eqref{locdds} is concisely denoted by $Z=(F^+,F^-)_f$.

\begin{figure}[H]
\centering 
\begin{overpic}[width=12cm]{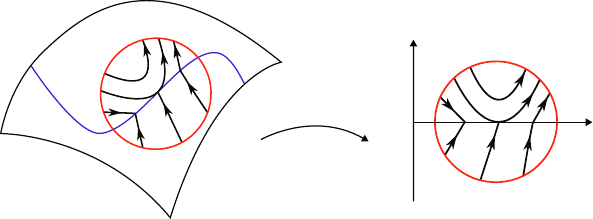}
\put(31,30){$U$}
\put(52,16.6){$\Phi$}
\put(1,15){$M$}
\put(96,15.5){$\Sigma$}
\put(40.25,25){$N$}
\put(92.5,7){$D$}
\put(20,25){$\CZ^+$}
\put(30,16){$\CZ^-$}
\put(83.5,22){$F^+$}
\put(83.5,10){$F^-$}
\put(41,9){$\CZ |_{U}\mapsto  Z=\Phi_* \CZ |_U$}
\end{overpic}
\vspace{0.5cm}
\caption{Local description of the piecewise smooth vector field $\eqref{dds}$ in $M$ by using local coordinates around points at the  switching manifold $N$.} \label{localdesc}
\end{figure}
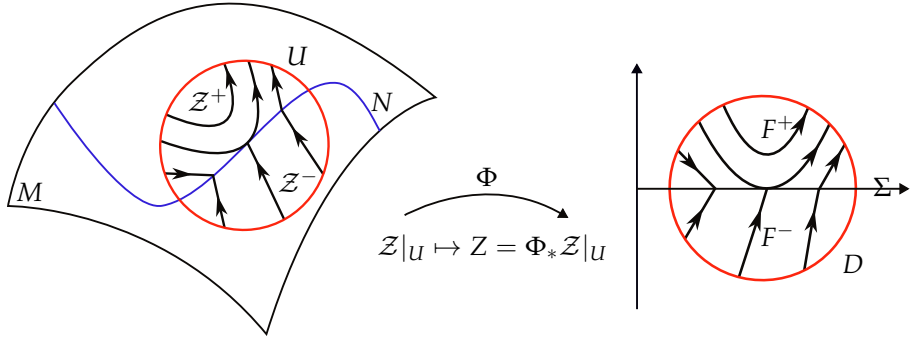

\begin{remark}\label{def-f}
The Filippov's convention for trajectories of \eqref{locdds} only depends on the zero set $\Sigma$ of the function $f$. However, in the next section, we are going to introduce a regularization process for Filippov vector fields, which will be a key tool for defining index in this context. For such a process, the expression of the function $f$ plays some role (see expression \eqref{regula}) and that is why we must carry $f$ in the Filippov vector field \eqref{locdds} instead of just $\Sigma$. It is important to anticipate that the index definition will not depend on $f$  (see Remark \ref{invar-f-phi}).
\end{remark}

In \cite{F}, Filippov established that the trajectories of \eqref{locdds} correspond to the solutions of the differential inclusion 
\begin{equation}
\label{eq:di}
\dot\x \in\CF_Z(\x),
\end{equation} where $\CF_Z$ is the following set-valued function
\[
\CF_Z(\x)=\left\{
\begin{array}{ll}
\{F^-(\x)\}&\text{if}\,\, f(\x)<0,\\
\{F^+(\x)\}&\text{if}\,\, f(\x)>0,\\
\left\{\dfrac{1+s}{2} F^+(\x)+\dfrac{1-s}{2} F^-(\x):\,s\in[-1,1]\right \}&\text{if}\,\, f(\x)=0.
\end{array}\right.
\]
The piecewise smooth vector field \eqref{dds} is called {\it Filippov vector field} when its local trajectories (i.e. trajectories of \eqref{locdds} for each $q\in N$) are ruled by the Filippov's convention. Filippov vector fields defined on compact manifolds have been addressed in \cite{MT19,NV21}.

The solutions of the differential inclusion \eqref{eq:di} have an easy geometrical interpretation which is fairly discussed in the research literature. 
In order to establish this geometrical interpretation, some regions on $\Sigma$ must be distinguished. Before, denote by $Fh$ the first Lie derivative of $h$ in  the direction of the vector field $F$, i.e. $Ff(\x)=\langle \nabla f(\x),F(\x)\rangle$. 

The {\it crossing region}, denoted by $\Sigma^c$, consists of the points $\x\in\Sigma$ such that $F^+f(\x)F^-f(\x)>0$. Notice that at a point $\x\in\Sigma^c$,  the solutions either side of the  switching manifold $\Sigma$, reaching $\x$, can be joined continuously, forming a solution that crosses $\Sigma$ (see Figure \ref{filippovelements}).

The {\it sliding region} (resp. {\it escaping region}), denoted by $\Sigma^s$ (resp. $\Sigma^e$), consists of the points $\x\in\Sigma$ such that $F^+f(\x)<0$ and $F^-f(\x)>0$ (resp. $F^+f(\x)>0$ and $F^-f(\x)<0$). Notice that, at a point $\x\in\Sigma^s$ (resp. $\x\in \Sigma^e$), both vector  $F^+(\x)$ and $F^-(\x)$ point inward (resp. outward) $\Sigma$ in such a way that the solutions either side of $\Sigma$, reaching $\x$, cannot be concatenate. Alternatively, for $\x\in \Sigma^{s,e}=\Sigma^s\cup \Sigma^e\subset N$, the solutions either side of $\Sigma$, reaching $\x,$ can be joined continuously to solutions that slide on $\Sigma^{s,e}$ following the so-called {\it sliding vector field} (see Figure \ref{filippovelements}):
\begin{equation}\label{slid}
Z^s(\x)= \dfrac{F^- f(\x) F^+(\x)- F^+ f(\x) F^-(\x)}{F^- f(\x) - F^+ f(\x) },\,\, \text{for} \,\, \x\in \Sigma^{s,e}.
\end{equation}
We notice that, for each $\x\in \Sigma^{s,e}$, $Z^s(\x)$ corresponds to the unique vector in the convex combination $\CF_Z(\x)$ that is tangent to $\Sigma$ at $\x$. Thus, any trajectory of the sliding vector field \eqref{slid} satisfies the differential inclusion \eqref{eq:di} and, therefore, is a trajectory of the Filippov vector field \eqref{locdds}. The sliding vector field \eqref{slid} associated with the Filippov vector field \eqref{locdds} defined on $D$ naturally induces a sliding vector field on $U\cap N$ associated with the Filippov vector field \eqref{dds} defined on the manifold $M$.

In what follows, we introduce the concept of singularities of Filippov vector fields (see, for instance, \cite{guardia2011generic}).

\begin{definition}\label{def:sing} Consider the Filippov vector field $Z$ given by \eqref{locdds}. We say that $\x_0\in D$ is a singularity of $Z$ if one of the following conditions hold:
\begin{itemize}
\item[(a)] $\x_0\in \Sigma^+$ (resp. $\x_0\in \Sigma^-$) and $F^+(\x_0)=0$ (resp. $F^-(\x_0)=0$). 
\item[(b)] $\x_0 \in \Sigma^s \cup \Sigma^e$ and $Z^s(\x_0) = 0$. 
\item[(c)] $\x_0\in \Sigma$ and $F^+f(\x_0) = 0$ or $F^-f(\x_0) = 0$. 
\end{itemize}
In case (a), if $\x_0\notin\Sigma$, then $\x_0$ it is just a singularity of one of the smooth vector fields, $F^+$ or $F^-$. Otherwise, it is called {\bf boundary equilibrium}. In  case (b), $\x_0$ is called {\bf pseudo-equilibrium}. In case (c), $\x_0$ is called {\bf tangential singularity}. Finally, a  singularity $\x_0$ of $Z$ is called {\bf $\Sigma$-singularity} if it belongs to $\Sigma$.
\end{definition}
Any point that does not satisfies the definition above is called {\bf regular}.

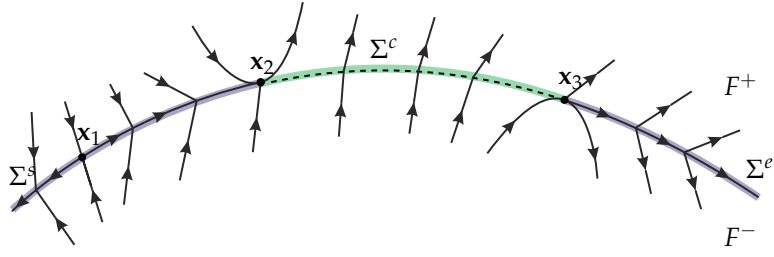
\begin{figure}[H]
\centering 
\begin{overpic}[width=10cm]{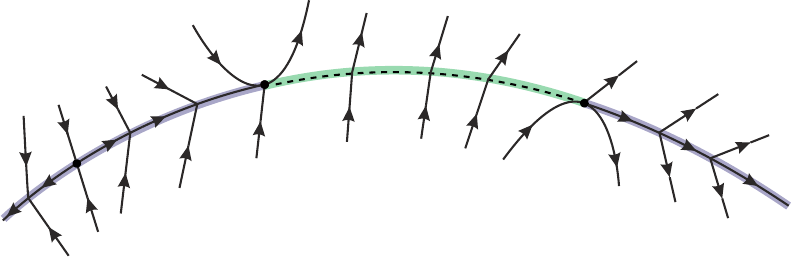}
\put(9,14){$\x_1$}
\put(32,23.5){$\x_2$}
\put(73,21.5){$\x_3$}
\put(0,8){$\Sigma^s$}
\put(98,9){$\Sigma^e$}
 \put(48,25){$\Sigma^c$}
\put(95,20){$F^+$}
\put(95,0){$F^-$}
\end{overpic}
\vspace{0.5cm}
\caption{Illustration of the Filippov's convetion. The points $\x_1\in\Sigma^s$ represents a pseudo-equilibrium and the points $\x_2,\x_3\in\p\Sigma^c$ represent tangential singularities.}\label{filippovelements}
\end{figure}
     
The Definition \eqref{def:sing} can be naturally extended for the Filippov vector field \eqref{dds} defined on the manifold $M$ as follows.

\begin{definition}\label{def:sing-mani} Consider the Filippov vector field $\CZ$ given by \eqref{dds}. We say that $p_0\in M$ is a singularity of $\CZ$ if there exists a chart  $(U,\Phi)$ of $M$ around $p_0$ such that $\Phi(p_0)$ satisfies Definition \ref{def:sing}.
\end{definition}

\section{Index of Filippov vector fields}\label{sec:def}

This section is devoted to provide the definition for index of singularities of Filippov vector fields defined on 2-dimensional compact  manifolds. 

First, for  $F,G:D\subset \R^2\to\R^2$ vector fields and $\x\in D,$ satisfying $\det ( G(\x) | F(\x) )\neq0$, we define the following auxiliary function,
\begin{equation}\label{HFG}
H_{(F,G)}(\x):=\dfrac{\|F(\x)\|^2 - \langle F(\x),G(\x)\rangle}{
\det \big( G(\x) | F(\x) \big) }.
\end{equation}

We start by defining the index of the Filippov vector field \eqref{locdds} on a circle $\p B$, where $B$ is the closed ball $B=B_{r}(\x_0)=\{\x\in\R^2:\, \|\x-\x_0\|\leq r\}$. indexfilippov

\begin{definition}\label{indexfilippov}
Let $Z$ be the Filippov vector field given by \eqref{locdds} and  $B\subset D$ a closed ball such that $\p B$ does not contain any singularities of $Z$. The index of $Z$ on $\partial B$ is defined by
\begin{equation}\label{eq:index}
I_{\partial B}(Z):=\dfrac{1}{2\pi}\left( J_{\partial B}(Z) + \int_{\Gamma^+} \omega_W +\int_{\Gamma^-} \omega_W\right),
\end{equation}
 where 
 $\omega_W$ is the following usual differential $1$-form
\[
\omega_W:=\dfrac{-y}{x^2+y^2}dx+\dfrac{x}{x^2+y^2}dy,
\]
 $\Gamma^{\pm}=\{Z(\x),\, \x\in \partial B_{r}(\x_0)\cap\Sigma^{\pm}\}$, and  $J_{\partial B}(Z)=J_{\partial B}^+(Z)+J_{\partial B}^-(Z)$ with 
\[
J_{\partial B}^{\pm}(Z)=\pm\begin{cases}\tan^{-1}\left(H_{(F^+,F^-)}(\pm r,0)\right)-\tan^{-1}\left(H_{(F^-,F^+)}(\pm r,0)\right),& \det ( F^+(\pm r,0) | F^-(\pm r,0) )\neq0,\\
0,& \det ( F^+(\pm r,0) | F^-(\pm r,0) )=0.
\end{cases}
\]
\end{definition}

Notice that, if $Z$ is smooth, that is, $F^+=F^-$, the above definition reduces to the usual index definition for smooth vector fields. Indeed, in this case, $\det ( F^+(\pm r,0) | F^-(\pm r,0) )=0$ and, therefore, $J_{\partial B}(Z)=0$. Hence,
\[
I_{\partial B}(Z):=\dfrac{1}{2\pi}\left( \int_{\Gamma^+} \omega_W +\int_{\Gamma^-} \omega_W\right)=\dfrac{1}{2\pi}\int_{\Gamma} \omega_W,
\]
being $\Gamma=\{Z(\x),\, \x\in \partial B_{r}(\x_0)\}$, which coincides with the usual index of the smooth vector field $Z$ on $\p B$ (see \eqref{IgA} from the Appendix).

	\begin{remark}\label{rem:equiJ}Geometrically, the absolute value of $J_{\partial B}^{\pm}(Z)$ coincides with the smallest angle between the vectors $ F^+(\pm r,0) $ and $ F^-(\pm r,0) $. The sign of $ J_{\partial B}^{\pm}(Z) $ is determined by the relative position of these vectors. Specifically, $ J_{\partial B}^{\pm}(Z) > 0 $ provided that $ F^{\mp}(\pm r,0) $ reaches $ F^{\pm}(\pm r,0) $ through a counter-clockwise rotation by the angle $|J_{\partial B}^{\pm}(Z)|\neq 0$, which is equivalent to $\sgn(\det \big(F^{\mp}(\pm r,0)|F^{\pm}(\pm r,0)  \big))>0$. Conversely, $ J_{\partial B}^{\pm}(Z) < 0 $ provided that the rotation is clockwise, which is equivalent to $\sgn(\det \big(F^{\mp}(\pm r,0)|F^{\pm}(\pm r,0)  \big))<0$. Finally, $J_{\partial B}^{\pm}(Z) =0 $ provided that the angle between $F^+(\pm r,0) $ and $ F^-(\pm r,0) $ is zero, which is equivalent to  $\det ( F^+(\pm r,0) | F^-(\pm r,0) )=0$. It is important to note that an angle of $ \pi $ between these vectors (another case where $\det ( F^+(\pm r,0) | F^-(\pm r,0) )=0$) is impossible, as it would imply the existence of pseudo-equilibria (see Definition \ref{def:sing}) on $\partial B$. Analytically, we have that
\[
J_{\partial B}^{\pm}(Z)=\pm\sgn(\det \big( F^-(\pm r,0)|F^+(\pm r,0)  \big))\arccos\left(\dfrac{\big\langle F^+(\pm r,0) ,F^-(\pm r,0) \big\rangle}{\|F^+(\pm r,0) \|\,\|F^-(\pm r,0)\| }\right).
\]
The above alternative formula for $J_{\partial B}^{\pm}(Z)$ is proven in Proposition \ref{prop:equiJ}.
\end{remark}
		
The above definition is based on the fact that the index given by \eqref{eq:index} is invariant under Sotomayor-Texeira regularization (ST-regularization) \cite{ST}, see Theorem \ref{thm:invreg} below. Roughly speaking, a regularization of a  piecewise smooth vector field $Z$ is a 1-parameter family $Z_{\e}$ of $C^l$, $l\geq 0$, vector fields such that $Z_{\e}$ converges to $Z$ when $\e\to 0$. The ST-regularization is defined by
\begin{equation}\label{regula}
Z_{\e}(\x)=\dfrac{1+\phi_{\e}\circ f(\x)}{2}F^+(\x)+\dfrac{1-\phi_{\e}\circ f(\x)}{2}F^-(\x),\,\,\text{being}\,\, \phi_{\e}(s):=\phi(s/\e),
\end{equation}
where $\phi:\R\to\R$ is a $C^l$, $l\geq 0$, function which is $\CC^1$ for $s\in(-1,1)$  satisfying $\phi(s)=\sgn(s)$ for $|s|\geq1$ and $\phi'(s)>0$ for $s\in(-1,1)$. We call $\phi$ a {\it monotonic transition function}. In this paper we will always consider smooth monotonic transition function, i.e. $r\geq 1$. It is worth noting that there are other methods for regularizing piecewise smooth vector fields, such as using non-monotonic transition functions as discussed in \cite{SMN18, NovJef15}. However, the ST-regularization has a unique connection with Filippov’s convention. As demonstrated in \cite{TexeiraSilva12}, the ST-regularization of Filippov vector fields leads to singular perturbation problems, where the resulting reduced dynamics are conjugated to the sliding dynamics \eqref{slid}.

\begin{remark}\label{globalreg}
Originally, the ST-regularization was introduced for Filippov vector fields defined on Euclidean spaces. In \cite{panazzolo2017regularization}, the ST-regularization was extended to Filippov vector fields defined on a smooth manifold $M$ with $1$-dimensional smooth switching manifold $N\subset M$, by  considering oriented 1-foliations on $M\setminus N$ (see \cite[Section 3.1]{panazzolo2017regularization}). We shall call such an extension by {\bf global ST-regularization}. Roughly speaking, a global ST-regularization of the Filippov vector field $\CZ$ given by \eqref{dds} is a 1-parameter family $\CZ_{\e}$ of smooth vector fields defined on $M$ for which there exists an atlas $\mathcal{A}=\{(U_{\alpha},\Phi_{\alpha}):\,\alpha\}$ satisfying that if $U_{\alpha}\cap N\neq\emptyset$, then the pushforward $(\Phi_{\alpha})_*\CZ_{\e}$ writes like \eqref{regula}. 
\end{remark}

In the sequel, we state the main result of this section  that provides the invariance  of Definition \ref{indexfilippov} under ST- regularization.

\begin{theorem}\label{thm:invreg}
Let $Z$ be the Filippov vector field given by \eqref{locdds}, $Z_{\e}$ its ST-regularization \eqref{regula}, and $B\subset D$ a closed ball such that $\p B$ does not contain any singularities of $Z$. Then, for $\e>0$ sufficiently small, $\p B$ does not contain any singularities of $Z_\e$ and $I_{\partial B}(Z)=I_{\partial B}(Z_{\e})$.
\end{theorem}

Theorem \ref{thm:invreg} is proven in Section \ref{sectiom:invariaceunderreg}.

\begin{remark}\label{invar-f-phi}
Theorem \eqref{thm:invreg} provides that the index of $Z_{\e}$, the ST-regularization of $Z$, along $\p B$ is invariant under the choice of the function $f$ that describes $\Sigma$ and under the choice of the transition function $\phi$.
\end{remark}

Now, we define the index of an isolated singularity of a planar Filippov vector field (see Definition \ref{def:sing}). 
                
\begin{definition}\label{indexp}
Let $Z$ be the  Filippov vector field given by \eqref{locdds},  $\x_0$ an isolated singularity of $Z$, and $r>0$ such that $\x_0$ is the unique singularity in $B= B_{r}(\x_0)\subset D.$ The index of $Z$ at $\x_0$ is defined as $I_{\x_0}(Z):=I_{\partial B}(Z).$ 
\end{definition}

The following lemma, together with Theorem \ref{thm:invreg}, will provide that the above index of a Filippov vector field at an isolated singularity is well defined. 
\begin{lemma}[{\cite[Proposition 6]{sotomayor2002structurally}}]\label{analucia1}
Let $Z$ be the Filippov vector field given by \eqref{locdds},  $Z_{\e}$ its ST-regularization \eqref{regula}, and  $\x_0$ a regular point of $Z$. Then, there exist a neighborhood $V_{\x_0}$ of $\x_0$ and $\e_{\x_0}>0$ such that $0\notin Z_{\e}(V_{\x_0})$ for every $\e\in(0,\e_{\x_0})$.
\end{lemma}

Lemma \ref{analucia1} will also be a key result in Section \ref{sec:prop} to establishing some properties for the index of Filippov vector fields.

The next result shows that the index $I_{\x_0}(Z)$ does not depend on the radius $r$ of the ball $B_{r}(\x_0)\subset D$, as long as $\x_0$ is the unique singularity in $B_{r}(\x_0)$. This ensures that the index of a Filippov vector field at a singularity is well defined.
 
\begin{proposition}\label{prop:invrad}
Let $Z$ be the Filippov vector field given by \eqref{locdds},  $\x_0$ an isolated singularity of $Z$, and $r_1>r_0>0$ such that $x_0$ is the unique singularity of $Z$ inside $B_1= B_{r_1}(\x_0)\subset D$. Then, $I_{\partial B_0}(Z)=I_{\partial B_1}(Z),$ where $B_0=B_{r_0}(\x_0).$
\end{proposition}

\begin{proof}
From Theorem \ref{thm:invreg}, there exists $\ov\e>0$ such that $I_{\partial B_0}(Z)=I_{\partial B_0}(Z_{\e})$ and $I_{\partial B_1}(Z)=I_{\partial B_1}(Z_{\e})$ for every $\e\in(0,\ov\e)$.  We claim that there exists $\e^*\in (0,\ov\e)$ and a continuous deformation $B_s\subset D,$ $s\in[0,1]$, of $B_0$ into $B_1$   such that  $0\notin Z_{\e^*}(\p B_s)$ for every $s\in[0,1]$. Indeed, let $K=\{\x\in\R^2:\,r_0\leq \|\x-\x_0\|\leq r_1\}\subset D.$  Notice that $B_s$ can be chosen in such a way that $\p B_s\subset K$ for every $s\in[0,1]$. By taking into account that $K$ is compact and that $Z\big|_K$ has only regular points,  Lemma \ref{analucia1} provides $\e^*\in (0,\ov\e)$ such that $0\notin Z_{\e^*}(K)$ and the claim follows. Hence, since $Z_{\e^*}$ is a smooth vector field, it follows that $I_{\partial B_0}(Z_{\e^*})=I_{\partial B_1}(Z_{\e^*})$ (see Proposition \ref{appendixpropcurvas} from the Appendix), which implies that $I_{\partial B_1}(Z)=I_{\partial B_0}(Z).$
\end{proof}

In what follows, we extend the Definition \ref{indexp} to singularities of Filippov vector fields defined on a $2$-dimensional manifold $M$ (see Definition \ref{def:sing-mani}). 

\begin{definition}\label{indiceM} Let $\CZ$ be the Filippov vector field (defined on a $2$-dimensional compact manifold $M$) given by \eqref{dds}. Let $p_0 \in M$ be an isolated singularity of  $\CZ$ and $(U,\Phi)$ a chart of $M$ around $p_0$ such that $\Phi_* \CZ$ writes like \eqref{locdds}. The index of $\CZ$ at $p_0$ is defined as $I_{p_0}(\CZ):= I_{\Phi(p_0)}(\Phi_* \CZ).$
\end{definition}

The next result shows that the index $I_{p_0}(\CZ)$ does not depend  on the chart $(U,\Phi)$. This ensures that the index of a Filippov vector field (defined on a $2$-dimensional manifold $M$) at a singularity is well defined.

\begin{proposition}
Let $\CZ$ be the Filippov vector field given by \eqref{dds}. Let $p_0 \in M$ be an isolated singularity of  $\CZ$ and let $(U,\Phi)$ and $(V,\Psi)$ be two charts of $M$ around $p_0$ such that $\Phi(p_0)=\Psi(p_0)=\x_0$ and the pushforward vector fields $\Phi_* \CZ$ and $\Psi_* \CZ$ write like \eqref{locdds}. Then, $I_{\x_0}(\Phi_* \CZ)=I_{\x_0}(\Psi_* \CZ)$
\end{proposition}
\begin{proof}

Denote $D_{1}=\Phi(U\cap V)$ and $D_2=\Psi\circ\Phi^{-1}(D_1)$ and consider the Filippov vector fields 
\[
Z=(F^+,F^-)_f=:\Phi_* \CZ: D_1\to\R^2\,\,\,\text{and}\,\,\, W=(G^+,G^-)_g:=(\Psi_* \CZ)|_{D_2}:D_2\to \R^2.
\]
Define the diffeomorphism $\alpha:=\Psi\circ\Phi^{-1}:D_1\to D_2$ and notice that $W=\alpha_* Z$, $G^+=\alpha_* F^+$ and $G^-=\alpha_* F^-$. It will  be convenient for us to take $g=f\circ \alpha^{-1}$ which, from Remark \ref{def-f}, can be done without loss of generality.

Also, let $r_0>\bar{r}>r_1>0$ be such that $\x_0$ is the unique singularity of $Z$ and $W$ inside $B_0=B_{r_0}(\x_0)\subset D_1,$  $\alpha(B)\subset B_0$ where $B=B_r(\x_1)$, and $B_1=B_{r_1}(\x_0)\subset B$. From Definition \ref{indexp},
\begin{equation}\label{indexespushforward}
I_{\x_0}(\Phi_* \CZ)=I_{\p B}(Z)\,\,\, \text{and}\,\,\, I_{\x_0}(\Psi_* \CZ)=I_{\p B}(W).
\end{equation}

Now, let $Z_{\e}$ and $W_{\e}$ be the ST-regularizations, given by \eqref{regula}, of $Z$ and $W$, respectively. From Theorem \ref{thm:invreg}, and taking Remark \ref{invar-f-phi} into account, there exists $\bar{\e}>0$ such that $Z_{\e}$ and $W_{\e}$ do not vanish on $\p B$ and
\begin{equation}\label{indexesFW}
I_{\p B}(Z)=I_{\p B}(Z_{\e})\,\,\,\text{and}\,\,\,I_{\p B}(W)=I_{\p B}(W_{\e})
\end{equation}
for every $\e\in(0,\bar{\e}]$.  We claim that $W_{\e}=\alpha_* Z_{\e}$. Indeed, for $\y\in D_2$
\[
\begin{aligned}
\alpha_* Z_{\e}(\y)=&\dfrac{1+\phi_{\e}\circ f(\alpha^{-1}(\y))}{2}\alpha_*F^+(\y)+\dfrac{1-\phi_{\e}\circ f(\alpha^{-1}(\y))}{2}\alpha_*F^-(\y)\\
=&\dfrac{1+\phi_{\e}\circ g(\y)}{2}G^+(\y)+\dfrac{1-\phi_{\e}\circ g(\y)}{2}G^-(\y)\\
=&W_{\e}(\y).
\end{aligned}
\]

Finally, since $Z_{\e}$ and $W_{\e}$ are smooth vector fields, we have that
 $I_{\p B}(Z_{\e})=I_{\alpha(\p B)}(W_{\e})$, for $\e\in(0,\bar{\e}]$ (see Proposition \ref{inv-change} from the Appendix). From Lemma \ref{analucia1} and taking into account the compactness of $K=\{\x\in B:\,r_1\leq||\x-\x_0||\leq r_0\}$, we can choose $\e^*\in (0,\bar{\e})$ for which $Z_{\e^*}$ and $W_{\e^*}$ do not vanish on $K$. From the choice of $r_0,\bar{r},$ and $r_1$, we have that $\p B\subset K$ and $\alpha(\p B)\subset K$, thus they can be continuously deformed into each other without passing through a singularity. Therefore, $I_{\alpha(\p B)}(W_{\e})=I_{\p B}(W_{\e})$ (see Proposition \ref{appendixpropcurvas} from the Appendix), which implies that $I_{\p B}(Z_{\e})=I_{\p B}(W_{\e}).$ Hence, from \eqref{indexespushforward} and \eqref{indexesFW}, it follows that $I_{\x_0}(\Phi_* \CZ)=I_{\x_0}(\Psi_* \CZ)$.
\end{proof}

\section{Properties of the index for Filippov vector fields}\label{sec:prop}
In this section, we apply Theorem \ref{thm:invreg} together with Lemma \ref{analucia1} to extend the index properties of smooth vector fields (see the Appendix) to the index for Filippov vector fields established in Section \ref{sec:def}.

\begin{proposition}\label{prop:invpera}
Let  $Z$ be the Filippov vector given by \eqref{locdds} and $B\subset D$ a closed ball. If $Z$ has no singularities on $\p B$, then $I_{\partial B}(Z)\in\mathbb{Z}$.
\end{proposition}
\begin{proof}
Let $Z_{\e}$ be the ST-regularization \eqref{regula} of $Z$. Theorem \ref{thm:invreg} implies that there exists $\ov \e>0$ such that $I_{\partial B}(Z)=I_{\partial B}(Z_{\e})$ for every $\e\in(0,\ov \e)$. Since $Z_{\e}$ is a smooth vector field, it follows that $I_{\partial B}(Z_{\e})\in\mathbb{Z}$ (see the Appendix), which implies that $I_{\partial B}(Z)\in\mathbb{Z}$.
\end{proof}

\begin{proposition}\label{prop:invperb}
Let $Z$ be the Filippov vector given by \eqref{locdds} and $B\subset D$ a closed ball. If $Z$ has no singularities on $B$, then $I_{\partial B}(Z)=0$.
\end{proposition}

\begin{proof}
Let $Z_{\e}$ be the ST-regularization \eqref{regula} of $Z$. Theorem \ref{thm:invreg} implies that there exists $\ov \e>0$ such that $I_{\partial B}(Z)=I_{\partial B}(Z_{\e})$ for every $\e\in(0,\ov \e)$. By taking into account the compactness of $B$ and that $Z\big|_{B}$ has only regular points,  Lemma \ref{analucia1} implies that there exists $\e^*\in(0,\ov \e)$ such that $Z_{\e^*}\big|_{B}$ has only regular points. Since $Z_{\e^*}$ is a smooth vector field, it follows that $I_{\partial B}(Z_{\e^*})=0$ (see Proposition \ref{index0fornosingularity} from the Appendix), which implies that $I_{\partial B}(Z)=0$.
\end{proof}

\begin{proposition}\label{prop:invperc}
Let $Z$ be the Filippov vector field given by \eqref{locdds} and $B\subset D$ a closed ball. Assume that $Z$ has no singularities on $\p B$. Then,
$$\displaystyle\min_{\la\in[0,1]}\|(1-\la)F^+(\x)+\la F^-(\x)\|>0, \,\, \text{for every $\x\in \p B\cap\Sigma$}.$$
In addition, assume that the Filippov vector field  $\widetilde Z=(\widetilde F^+,\widetilde F^-)$ (defined on $D$ and given as \eqref{locdds}),
satisfies
\begin{itemize} 
\item $\|Z(\x)-\widetilde Z(\x)\|<\|Z(\x)\|,$ for every $\x\in \p B\setminus\Sigma$,
\item $|F^{\pm}f(\x)-\widetilde F^{\pm}f(\x)|<|F^{\pm}f(\x)|,$ for every $\x\in \p B\cap\Sigma$,
\item   $|Z^s(\x)-\widetilde Z^s(\x)|<|Z^s(\x)|,$ for every $\x\in \p B\cap\Sigma^s$, and
 \item $\displaystyle\|F^{\pm}(\x)-\widetilde F^{\pm}(\x)\|<\dfrac{1}{2}\min_{\la\in[0,1]}\|(1-\la)F^+(\x)+\la F^-(\x)\|,$ for every $\x\in \p B\cap\Sigma$.
\end{itemize} 
Then, $\widetilde Z$ has no singularities on $\p B$ and $I_{\partial B}(Z)=I_{\partial B}(\widetilde Z)$.
\end{proposition}
\begin{proof}
First, let us verify that
$$\displaystyle\min_{\la\in[0,1]}\|(1-\la)F^+(\x)+\la F^-(\x)\|>0, \,\, \text{for every $\x\in \p B\cap\Sigma$}.$$
Indeed, assume for a contradiction that there exist $\bar \x\in \partial B\cap \Sigma$ and $\bar{\lambda} \in [0,1]$ such that $\|(1-\bar{\lambda})F^+(\bar{\x})+\bar{\lambda}F^-(\bar{\x})\|=0,$ i.e. $(1-\bar{\lambda})F^+(\bar{\x})=-\bar{\lambda}F^-(\bar{\x})$. Since  $F^\pm(\bar{\x})\neq 0$, then $\bar{\lambda}\notin \{0,1\}.$ Therefore,
\begin{equation*}
\begin{array}{rll}
F^-(\bar{\x})=-\left(\dfrac{1-\bar{\lambda}}{\bar{\lambda}}\right)F^+(\bar{\x})~~
\Rightarrow ~~ \dfrac{F^-f(\bar{\x})}{F^+f(\bar{\x})}=-\left(\dfrac{1-\bar{\lambda}}{\bar{\lambda}}\right)<0,
\end{array}
\end{equation*}
This implies that $\bar \x\in\Sigma^s$ and, consequently, $\bar{\x}$ is a pseudo-equilibrium, i.e. $Z^s(\bar{\x})=0$ which contradicts the hypothesis. 

Now notice that, by hypothesis, $\|\widetilde Z(\x)\|>0$ for $\x\in \p B\setminus\Sigma,$  $|\widetilde F^{\pm}f(\x)|>0$ for $\x\in \p B\cap\Sigma$, and $|\widetilde Z^s(\x)|>0$ for $\x\in \p B\cap\Sigma^s$. This implies that $\widetilde Z$ has no singularities on $\p B$. 

Finally, let $Z_{\e}$ and $\widetilde Z_{\e}$  be the ST-regularizations of $Z$ and $\widetilde Z$, respectively.  Theorem \ref{thm:invreg} implies that there exists $\ov \e>0$ such that $I_{\partial B}(Z)=I_{\partial B}(Z_{\e})$ and $I_{\partial B}(\widetilde Z)=I_{\partial B}(\widetilde Z_{\e})$ for every $\e\in(0,\ov \e)$. In addition, from hypothesis and taking into account that $\phi(\R)\subset[-1,1]$, we get 
\begin{equation*}
\begin{array}{lll}
\|Z_{\e}(\x)-\widetilde Z_{\e}(\x)\|&=& \left\|\dfrac{1+\phi_{\e}(f(\x))}{2}\left(F^+(\x) - \widetilde F^+(\x)\right)+\dfrac{1-\phi_{\e}(f(\x))}{2}\left(F^-(\x)-\widetilde F^-(\x)\right) \right\|\vspace{0.5cm}\\
&\leq& \|F^+(\x) - \widetilde F^+(\x)\| + \|F^-(\x) - \widetilde F^-(\x)\|\vspace{0.5cm}\\
&<& \min_{\la\in[0,1]}\|(1-\la)F^+(\x)+\la F^-(\x)\|< \|Z_{\e}(\x)\|,
\end{array}
\end{equation*}
for every $\x\in \p B$ and $\e\in(0,\bar \e]$. Thus, since  $Z_{\e}$ and $\widetilde Z_{\e}$ are smooth vector fields, it follows that $I_{\partial B}(\widetilde Z_{\e})=I_{\partial B}(Z_{\e})$  for $\e\in(0,\bar \e]$ (see Proposition \ref{prop:pertsmooth} from the Appendix), which implies that $I_{\partial B}(\widetilde Z)=I_{\partial B}(Z)$.
\end{proof}

\begin{proposition}\label{prop:invperd}
Let $Z$ be the Filippov vector given by \eqref{locdds}, $B\subset D$ a closed ball, and  $Z_1(\cdot;\delta)=(F^+_1(\cdot;\delta),F^-_1(\cdot;\delta))_f$ be a continuous 1-parameter family of Filippov vector fields (defined on $D$ and given as \eqref{locdds}) such that $Z_1(\cdot;0)$ vanishes identically, that is $F^{\pm}_1(\cdot;0)=(0,0)$.  Consider the Filippov vector field $\widetilde Z(\cdot;\delta)=Z+ Z_1(\cdot;\delta)$. Then, there exists $\bar \delta>0$ such that $I_{\partial B}(Z)=I_{\partial B}(\widetilde Z(\cdot;\delta))$ for every $\delta\in (0,\bar\delta)$.
\end{proposition}
\begin{proof}
First, notice that $\widetilde Z(\cdot,\delta)=(\widetilde F^+(\cdot,\delta),\widetilde F^-(\cdot,\delta))_f$, where $\widetilde F^{\pm}(\cdot,\delta)=F^{\pm}+F_1^{\pm}(\cdot,\delta)$. 
Since, $F_1^{\pm}(\x;\de)\to (0,0)$, uniformly for $\x\in B,$  as $\delta\to0$ we obtain that:
 $\displaystyle\|Z(\x)-\widetilde Z(\x;\delta)\|\to0$,  uniformly for $\x\in \partial B\setminus\Sigma$,  as $\delta\to0$;  $|F^{\pm}f(\x)-\widetilde F^{\pm}f(\x;\delta)|\to0$ and  $\displaystyle\|F^{\pm}(\x)-\widetilde F^{\pm}(\x;\delta)\|\to0$,  for $\x\in\partial B\cap \Sigma$,  as $\delta\to0$; and $|Z^s(\x)-\widetilde Z^s(\x;\delta)|\to0$,  for $\x\in\partial B\cap \Sigma^s$, as $\delta\to0$. Thus, taking into account that $Z$ has no singularities on $\p B$, we conclude that there exists $\bar \de>0$ for which the hypotheses of Proposition \ref{prop:invperc} hold for $\widetilde Z(\cdot;\delta)$, for $\delta\in(0,\bar \delta)$, which implies that $I_{\partial B}(Z)=I_{\partial B}(\widetilde Z(\cdot;\delta))$ for every $\delta\in (0,\bar\delta)$. 
\end{proof}

\begin{proposition}\label{prop:invpere}
Let $Z(\cdot;\la),$ $\la\in[0,1]$, be a continuous homotopy between Filippov vector fields given as \eqref{locdds} and $B\subset D$ a closed ball such that  $Z(\cdot;\la)$ has no singularities on $\p B,$ for every $\la\in[0,1]$. Then, $I_{\partial B}(Z(\cdot;\la))$ is constant on $\la\in[0,1]$.\end{proposition}
\begin{proof}
For each $\la_0\in[0,1],$ by taking $Z=Z(\cdot;\la_0)$ and $Z_1(\cdot;\delta)=Z(\cdot;\la_0+\delta)-Z(\cdot;\la_0)$, Proposition \ref{prop:invperd} provides a neighborhood $J_{\la_0}\subset[0,1]$ of $\la_0$ such that $I_{\partial B}(Z(\cdot;\la))=I_{\partial B}(Z(\cdot;\la_0))$ for every $\la\in J_{\la_0}$, which implies that the map $\la\in[0,1]\mapsto I_{\partial B}(Z(\cdot,\la))$ is continuous. Since, from Proposition \ref{prop:invpera}, it is an integer-valued map, we conclude that $I_{\partial B}(Z(\cdot,\la))$
 is constant on $\la\in[0,1]$.
\end{proof}

\begin{proposition}\label{prop:somaindexsingu}
Let $Z$ be the Filippov vector field given by \eqref{locdds} and $B\subset D$ a closed ball such that $Z$ has no singularities on $\p B$. Assume that $Z$ has finitely many singularities inside $B$, $\x_1,\x_2,\ldots,\x_n$. Then, 
\[
I_{\partial B}(Z)=\sum_{i=1}^n I_{\x_i}(Z).
\]
\end{proposition}
\begin{proof}
Since each singularity $\x_i,~  i\in\{1, \ldots, n\},$ is isolated, there exist $r_i, ~ i\in\{1, \ldots, n\},$ small enough, such that $\x_i$ is the unique singularity inside $B_i=B_{r_i}(\x_i)$ for $i\in\{1,\ldots, n\},$ and $B_j\cap B_i = \emptyset,$ for every $i \neq j.$ From Theorem \ref{thm:invreg} and Definition \ref{indexp}, there exists $\bar{\e}>0$ such that
\[I_{\partial B}(Z)= I_{\partial B}(Z_\e) \text{ and } I_{\x_i}(Z)=I_{\partial B_i}(Z)=I_{\partial B_i}(Z_\e) \text{ for } i\in\{1,\ldots, n\} \text{ and } \e \in (0,\bar{\e}].
\]
Since  $Z_\e$ is a smooth vector field and $B= \ov{B\setminus \cup_i B_i}\cup (\cup_i B_i)$, Proposition \ref{appendixprop1} from the Appendix provides
\[I_{\partial B}(Z)=I_{\partial B}(Z_\e)=I_{\partial(B\setminus \cup_i B_i)}(Z_\e)+ \sum_{i=1}^n I_{\partial B_i}(Z_\e)=I_{\partial (B\setminus \cup_i B_i)}(Z_\e)+ \sum_{i=1}^n I_{\x_i}(Z), \text{ for } \e \in (0,\bar{\e}]. 
\]
Now, Lemma \ref{analucia1} provides the existence of $\e^* \in (0,\bar{\e}]$ such that $Z_{\e*}$ has no singularities in $\ov{B\setminus \cup_i B_i}$.
Thus, from Proposition \ref{index0fornosingularity} from the Appendix, we have that $
I_{\partial (B\setminus \cup_i B_i)}(Z_\e^*)=0$
which implies that $$I_{\partial B}(Z)=\sum_{i=1}^n I_{\x_i}(Z).$$
\end{proof}

We conclude this section by providing an alternative formula for $J_{\partial B}^{\pm}(Z)$ (see Remark \ref{rem:equiJ}).
\begin{proposition}\label{prop:equiJ} 
Let $Z$ be the Filippov vector field given by \eqref{locdds} and  $B\subset D$ a closed ball such that $\p B$ does not contain any singularities of $Z$. Let $J_{\partial B}^{\pm}(Z)$ be given by 
Definition \ref{indexfilippov}. Then, the following relationship holds:
\begin{equation}\label{equiJ}
J_{\partial B}^{\pm}(Z)=\pm\sgn(\det \big( F^-(\pm r,0)|F^+(\pm r,0)  \big))\arccos\left(\dfrac{\big\langle F^+(\pm r,0) ,F^-(\pm r,0) \big\rangle}{\|F^+(\pm r,0) \|\,\|F^-(\pm r,0)\| }\right).
\end{equation}
	\end{proposition}
	\begin{proof} 
	 Denote $J^{\pm}=J_{\partial B}^{\pm}(Z)$. If $\det ( F^+(\pm r,0) | F^-(\pm r,0) )=0$, then $J_{\partial B}^{\pm}(Z)=0$ and, therefore, the relationship \eqref{equiJ} holds. Thus, assume that $\det ( F^+(\pm r,0) | F^-(\pm r,0) )\neq0$. In the sequel, we will proceed with the proof for  $J^-$.  The argument for  $J^+$  follows analogously. 
	 
Notice that  $J^-\in(-\pi,\pi)$ and, therefore, $|J^-|\in[0,\pi)$. In addition, it is known that $\arctan(u)+\arctan(v)\in(-\pi/2,\pi/2)$ provided that $u\,v<1$; $\arctan(u)+\arctan(v)\in\{-\pi/2,\pi/2\}$ provided that $u\,v=1$; and $\arctan(u)+\arctan(v)\in(-\pi,-\pi/2]\cup[\pi/2,\pi)$ provided that $u\,v>1$. We are going to use this information to determine the sign of $J^-$. Recall that
\[
J^-=\tan^{-1}\left(-H_{(F^+,F^-)}(\pm r,0)\right)+\tan^{-1}\left(H_{(F^-,F^+)}(\pm r,0)\right).
\]
Taking into account that $\det(u|v)^2+\langle u,v\rangle^2=\|u\|^2\|v\|^2$ for any vectors $u,v\in\R^2$, it follows that
\[
c:=-H_{(F^+,F^-)}(-r,0)H_{(F^-,F^+)}(-r,0)=1-\dfrac{|F^+-F^-|^2}{\det \big( F^+(-r,0) | F^-(-r,0) \big)^2}\big\langle F^+(-r,0) ,F^-(-r,0) \big\rangle.
\]

Also, let
	\[
	\theta^-=\arccos\left(\dfrac{\big\langle F^+(-r,0) ,F^-(-r,0) \big\rangle}{\|F^+(-r,0) \|\,\|F^-(-r,0)\| }\right)\in[0,\pi].
	\]

First, if $\big\langle F^+(-r,0) ,F^-(-r,0) \big\rangle=0$, then $c=1$ and, therefore, $|J^-|=\pi/2=\theta.$ Moreover, in this case,
\[
J^-=\tan^{-1}\left(s\right)+\tan^{-1}\left(1/s\right),\quad\text{where}\quad s=\dfrac{|F^+(-r,0)|^2}{\det \big( F^+(-r,0) | F^-(-r,0) \big)},
\]
and, consequently, $\sgn(J^-)=\sgn(s)=\sgn(\det \big( F^+(-r,0) | F^-(-r,0) \big))$. Hence, the relationship \eqref{equiJ} holds.

Now, assume that $\big\langle F^+(-r,0) ,F^-(-r,0) \big\rangle\neq0$.  Using the trigonometric identities about tangent of arccosine and tangent of the sum of arctangents, one gets that
\begin{equation}\label{tanJ}
\tan J^-=\dfrac{\det \big( F^+(-r,0) | F^-(-r,0) \big)}{\big\langle F^+(-r,0) ,F^-(-r,0) \big\rangle},\quad\text{and}\quad
\tan\theta=\dfrac{|\det \big( F^+(-r,0) | F^-(-r,0) \big)|}{\big\langle F^+(-r,0) ,F^-(-r,0) \big\rangle},
\end{equation}
which implies that $|\tan\theta|=|\tan |J^-||.$ In addition, from \eqref{tanJ}, 
\[
\sgn(\tan J^-)=\sgn\big(\det ( F^+(-r,0) | F^-(-r,0) )\big)\sgn\big(\big\langle F^+(-r,0) ,F^-(-r,0) \big\rangle\big).
\]

Hence, if $\big\langle F^+(-r,0) ,F^-(-r,0) \big\rangle>0$, then $c<1$ and, thus, $J^-\in (-\pi/2,\pi/2)$. Consequently, 
\[
\sgn(J^-)=\sgn(\tan J^-)=\sgn\big(\det ( F^+(-r,0) | F^-(-r,0) )\big).
\]
Moreover, in this case, $\tan\theta>0$, implying that $\theta\in(0,\pi/2)$ and, therefore, $|J^-|=\theta$.

Now, if $\big\langle F^+(-r,0) ,F^-(-r,0) \big\rangle<0$, then $c>1$ and, thus, $J^-\in (-\pi,-\pi/2]\cup[\pi/2,\pi)$. Consequently, 
\[
\sgn(J^-)=-\sgn(\tan J^-)=\sgn\big(\det ( F^+(-r,0) | F^-(-r,0) )\big).
\]
Moreover, in this case, $\tan\theta<0$, implying that $\theta\in(\pi/2,\pi)$ and, therefore, $|J^-|=\theta$. It concludes the proof for $J^-$. The proof for $J^+$ is analogous. \end{proof}

\section{Indices of generic $\Sigma$-singularities}\label{sec:generic}

This section is devoted to compute the indices of the generic $\Sigma$-singularities of planar Filippov vector fields. Recall, from \cite{guardia2011generic}, that a generic $\Sigma$-singularity corresponds either to hyperbolic pseudo-node, to a hyperbolic pseudo-saddle, or to a regular-fold tangential singularity. In what follows, we define precisely these singularities.

\begin{definition}\label{def:pseudo.singu}
	A  pseudo-equilibrium $p \in \Sigma$  of $Z$ is said to be {\bf hyperbolic} provided that $(Z^s)'(p)\neq  0.$ Furthermore, it is called
	\begin{itemize}
		\item {\bf pseudo-saddle} if $p\in \Sigma^s$ (resp. $p\in \Sigma^e$)  and
		$(Z^s)'(p)>  0$ (resp. $(Z^s)'(p)<  0$); 
		\item {\bf pseudo-node} if $p\in \Sigma^s$ (resp. $p\in \Sigma^e$) and $(Z^s)'(p)<0$ (resp. $(Z^s)'(p)> 0$).
	\end{itemize}
\end{definition}

\begin{definition} A regular-fold tangential singularity is a point $p \in \Sigma$ satisfying one of the following properties: 
	\begin{itemize}
		\item $F^+ f (p) = 0$ and $(F^+)^2 f (p) \neq 0$ and $F^- f (p) \neq 0$. In this case, we say that the regular-fold point is visible if $(F^+)^2 f (p) > 0$ and invisible if
		$(F^+)^2 f (p) < 0$. 
		\item $F^- f (p) = 0$ and $(F^-)^2 f (p) \neq 0$ and $F^+ f (p) \neq 0$. In this case, it is visible provided $(F^-)^2 f (p) < 0$ and invisible provided $(F^-)^2 f (p) > 0$. 
	\end{itemize}
\end{definition}

The strategy outlined in this section for calculating the indices of generic $\Sigma$-singularities involves regularizing the Filippov vector field $Z$ within a neighborhood of the singularity. This is followed by using the invariance of the index under regularization (Theorem \ref{thm:invreg}), along with the following lemmas:

\begin{lemma}[{\cite[Proposition 8]{sotomayor2002structurally}}]\label{analucia2}
	Let $p$ be a hyperbolic pseudo-equilibrium of $Z$ as given in \eqref{locdds}. There exists a neighbourhood $V$ of $p$ and $\e_0>0$ such that for $0<\e\leq \e_0,$ $Z_{\e}$ has a unique critical point $p_\e$ in  $V.$ If $p$ is a pseudo-saddle then $p_\e$ is a saddle. If $p$ is a pseudo-node then $p_\e$ is a node.
\end{lemma}

\begin{lemma}[{\cite[Proposition 9]{sotomayor2002structurally}}]\label{analucia3}
	Let $Z$ be a Filippov vector field, as given in \eqref{locdds},  $Z_{\e}$ its regularization, and  $p$ a regular-fold point of $Z$. Then, there exist a neighborhood $V_p$ of $p$ and $\e_p>0$ such that $0\notin Z_{\e}(V_p)$ for every $\e\in(0,\e_p)$.
\end{lemma}

\begin{proposition}\label{prophiperbolic}
	Let $p$ be a generic $\Sigma-$singularity of a Filippov vector field $Z.$ Then $I_p(Z)=0$ provided that $p$ is a regular-fold tangential singularity; $I_p(Z)=1$ provided that $p$ is a hyperbolic pseudo-node; and $I_p(Z)=-1$ provided that $p$ is a hyperbolic pseudo-saddle.
\end{proposition}
\begin{proof}
	Let $p \in \Sigma$ be a regular-fold. Then by Lemma \ref{analucia3}, there is a neighborhood $V$ of $p$ and $\e_0 > 0$ such that for $0 < \e \leq \e_0$, $Z_{\e}$ has no critical points in $V$. Take a neighborhood $B \subset V$ of $p,$ by Theorem \ref{thm:invreg},  $I_{\partial B}(Z)=I_{\partial B}(Z_{\e})= 0$ since $Z_{\e}$ has no singularities (see the Appendix). Then, $I_{\partial B}(Z)=I_p(Z)=0$.
	
Now let $p \in \Sigma$ be a hyperbolic pseudo-equilibrium. By Lemma \ref{analucia2}, there exists a neighborhood $V$ of $p$ and $\e_0 > 0$ such that for $0 < \e \leq \e_0$, $Z_{\e}$ has a unique critical point $p_{\e}$ near $p$ which is a hyperbolic saddle if $p$ is a pseudo-saddle for $Z^s$, or a hyperbolic node if $p$ is a pseudo-node for $Z^s$. By Theorem \ref{thm:invreg}, if $p$ is a pseudo-node, then the index is $1$ whereas if $p$ is a pseudo-saddle the index is $-1.$
\end{proof}

\section{Indices of some codimension-1  $\Sigma$-singularities}\label{sec:cod1}

This section is devoted to compute the indices of some codimension-1 $\Sigma$-singularities of planar Filippov vector fields. From \cite{guardia2011generic}, a codimension-1 $\Sigma$-singularity $p$ can be classified as:

\begin{itemize}
\item {\bf Fold-fold tangential singularity:} if $p\in \Sigma$ satisfies $F^+f(p)=F^-(p)=0$, $(F^+)^2 f(p)\neq0$, and $(F^-)^2 f(p)\neq0$.
\item {\bf Regular-cusp tangential singularity:} if $p\in \Sigma$ satisfies $F^+f(p)=(F^+)^2 f(p)=0$, $(F^+)^3 f(p)\neq 0$, and $F^-f(p)\neq0$ (or, equivalently, $F^-f(p)=(F^-)^2 f(p)=0$, $(F^-)^3 f(p)\neq 0$, and $F^+f(p)\neq0$).
\item {\bf Saddle-node pseudo-equilibrium:} if $p\in\Sigma^{s,e}$ and satisfies $Z^s(p)=(Z^s)'(p)=0$, and $(Z^s)''(p)\neq0$.
\item {\bf Codimension-1 boundary equilibrium:} if $p\in \Sigma$ is a hyperbolic singularity of $F^+$ (resp. $F^-$), $F^-f(p)\neq0$ (resp. $F^+f(p)\neq0$), and $(Z^s)'(p)\neq0$.
\end{itemize}

In this section, we will compute the indices of fold-fold tangential singularities, regular-cusp tangential singularities, and saddle-node pseudo-equilibria. The indices of boundary equilibria will be addressed in future work. Our approach involves perturbing the vector field within a small neighborhood of each singularity to either remove it or unfold it into generic singularities. We then apply the invariance of the index under small perturbations (Proposition \ref{prop:invperd}), and either Proposition \ref{prop:invperb} or the known indices of generic $\Sigma$-singularities (Proposition \ref{prophiperbolic}) along with Proposition \ref{prop:somaindexsingu}.

\subsection{Indices of fold-fold tangential singularities}
Let  $p\in\Sigma$ be a fold-fold tangential singularity of a Filippov vector field $Z=(F^+,F^-)$ and
	denote 
	\begin{equation}\label{k}
	k=\sgn(F^-_1(0,0)) \dfrac{1}{|F_1^+(0,0)|}\dfrac{\partial F_2^+}{\partial x} (0,0)-\sgn(F^+_1(0,0)) \dfrac{1}{|F_1^-(0,0)|}\dfrac{\partial F_2^-}{\partial x} (0,0).
	\end{equation}
We shall see that $k=0$ increases the codimension of the bifurcation, so that we shall assume $k\neq0$. The next proposition provides the index of $p$ depending on the configuration of the fold-fold tangential singularity (see Figure \ref{fig:dd}).
	\begin{figure}[h!]
		\centering
		\begin{subfigure}{.25\textwidth}
			\includegraphics[width=.95\linewidth]{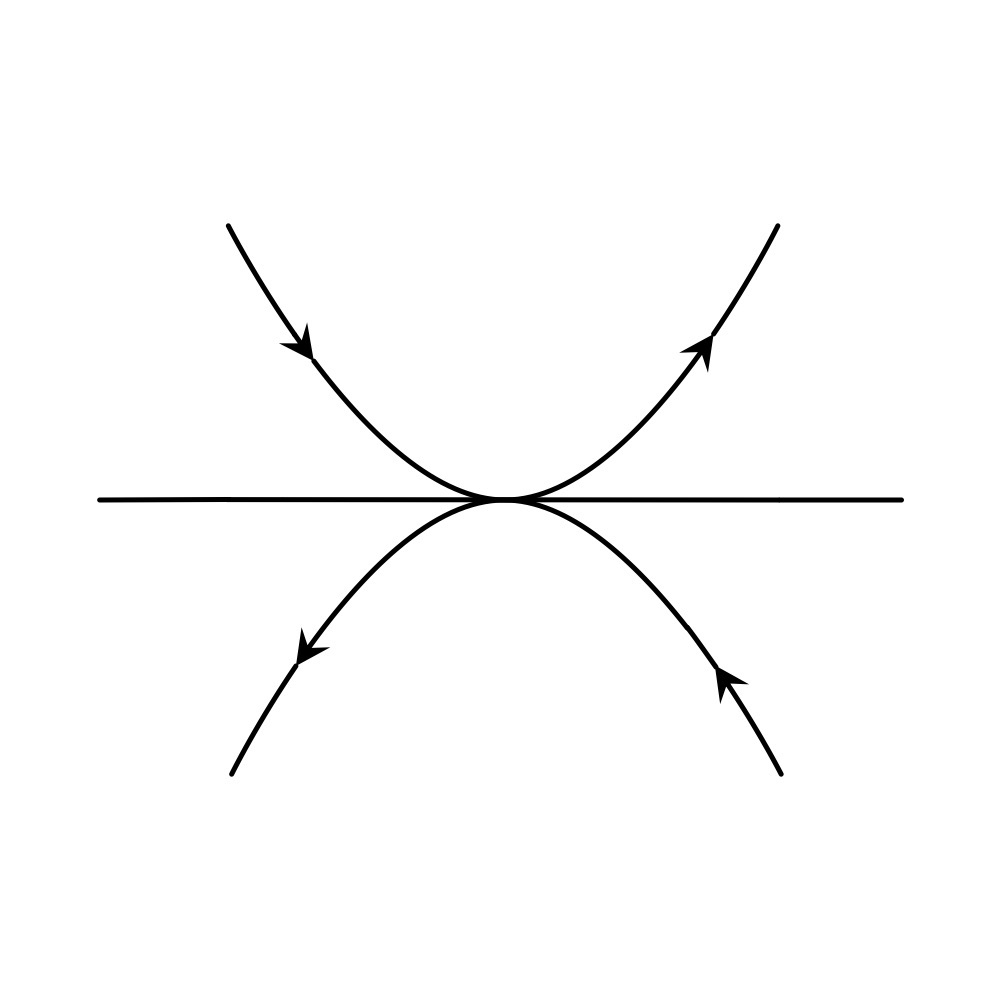}
			\caption{Index equal to $-1$}
			\label{fig:dd1}
		\end{subfigure}
		\begin{subfigure}{.25\textwidth}
			\includegraphics[width=.95\textwidth]{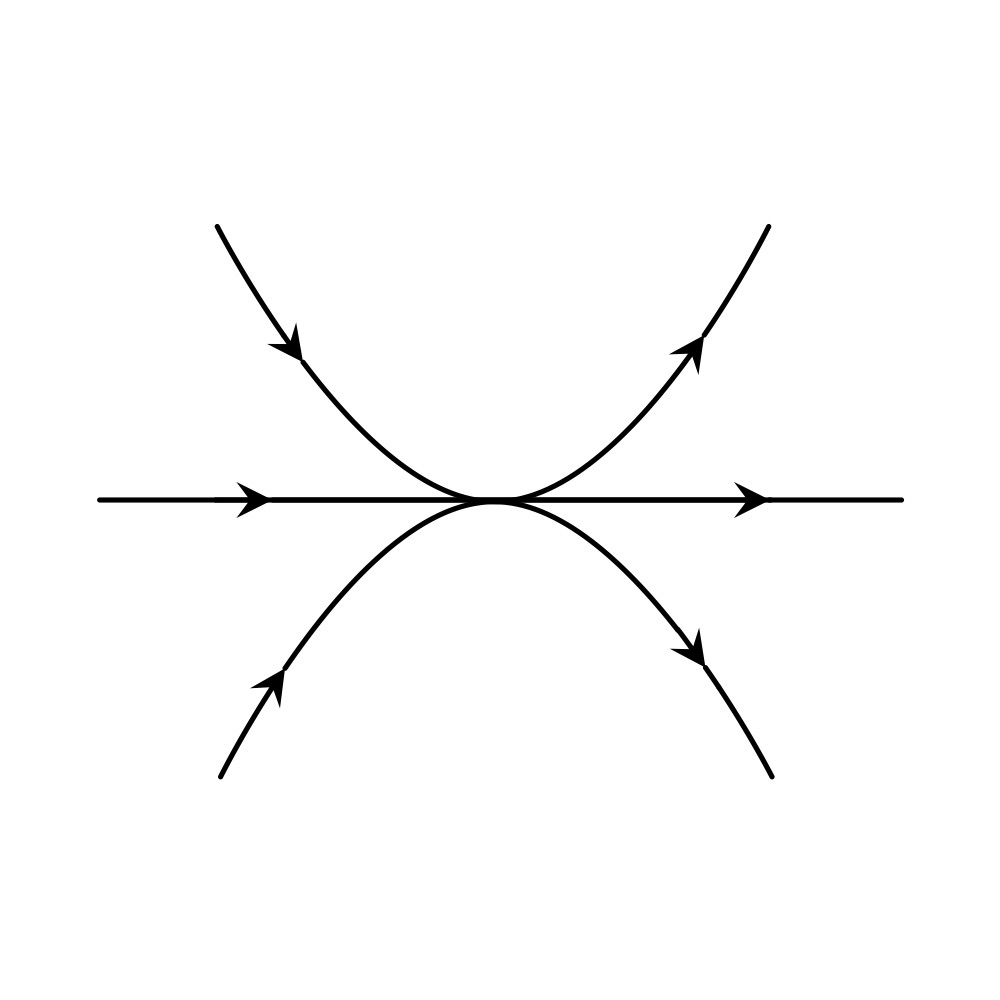}
			\caption{Index equal to $0$}
			\label{fig:dd2}
		\end{subfigure}\\
		\begin{subfigure}{.25\textwidth}
			\includegraphics[width=.95\linewidth]{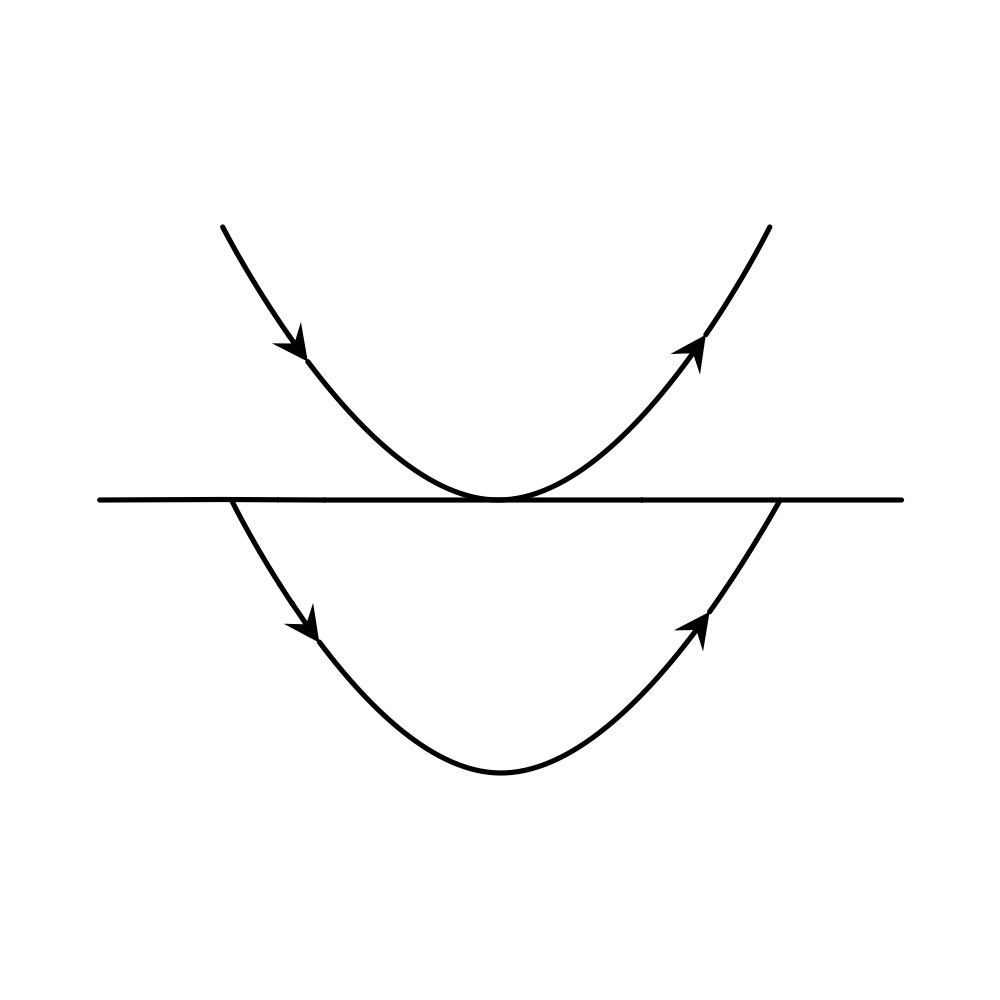}
			\caption{Index equal to $0$}
			\label{fig:dd3}
		\end{subfigure}%
		\begin{subfigure}{.25\textwidth}
			\includegraphics[width=.95\textwidth]{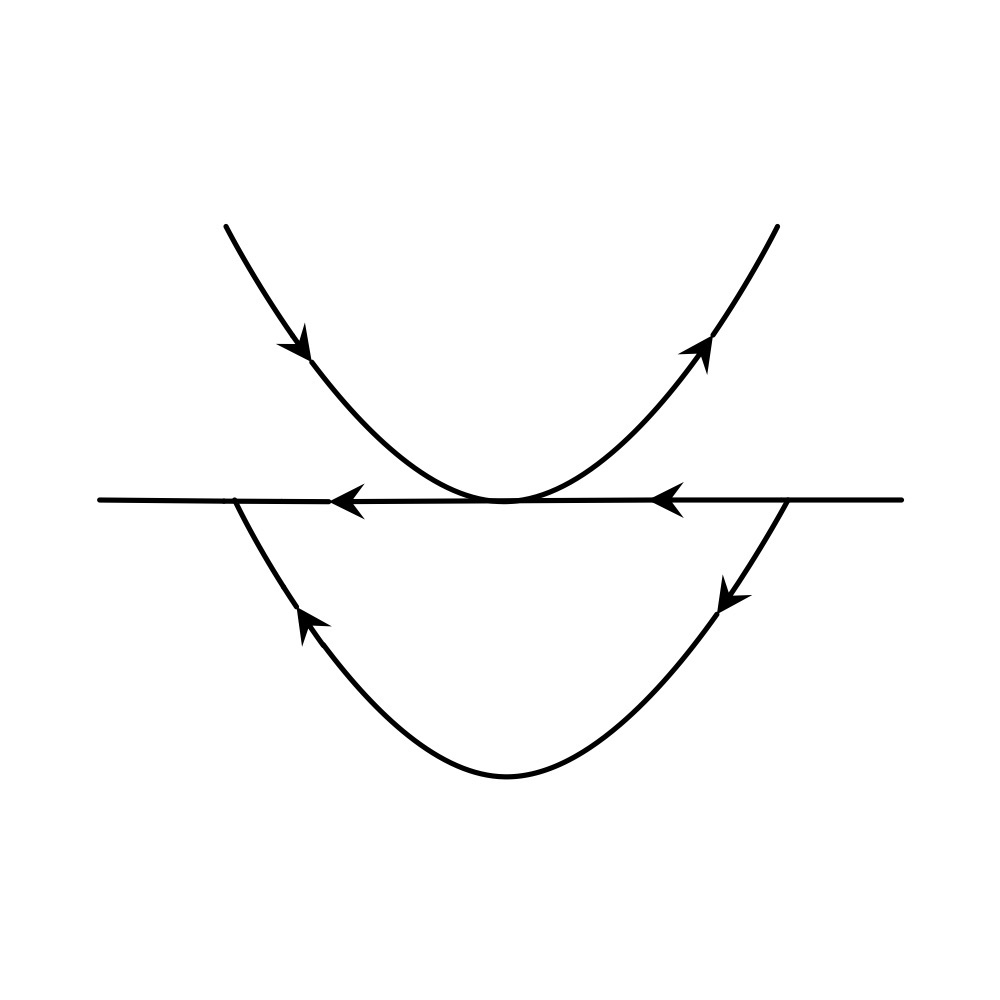}
			\caption{Index equal to $-1$}
			\label{fig:dd4}
		\end{subfigure}
		\begin{subfigure}{.25\textwidth}
			\includegraphics[width=.95\linewidth]{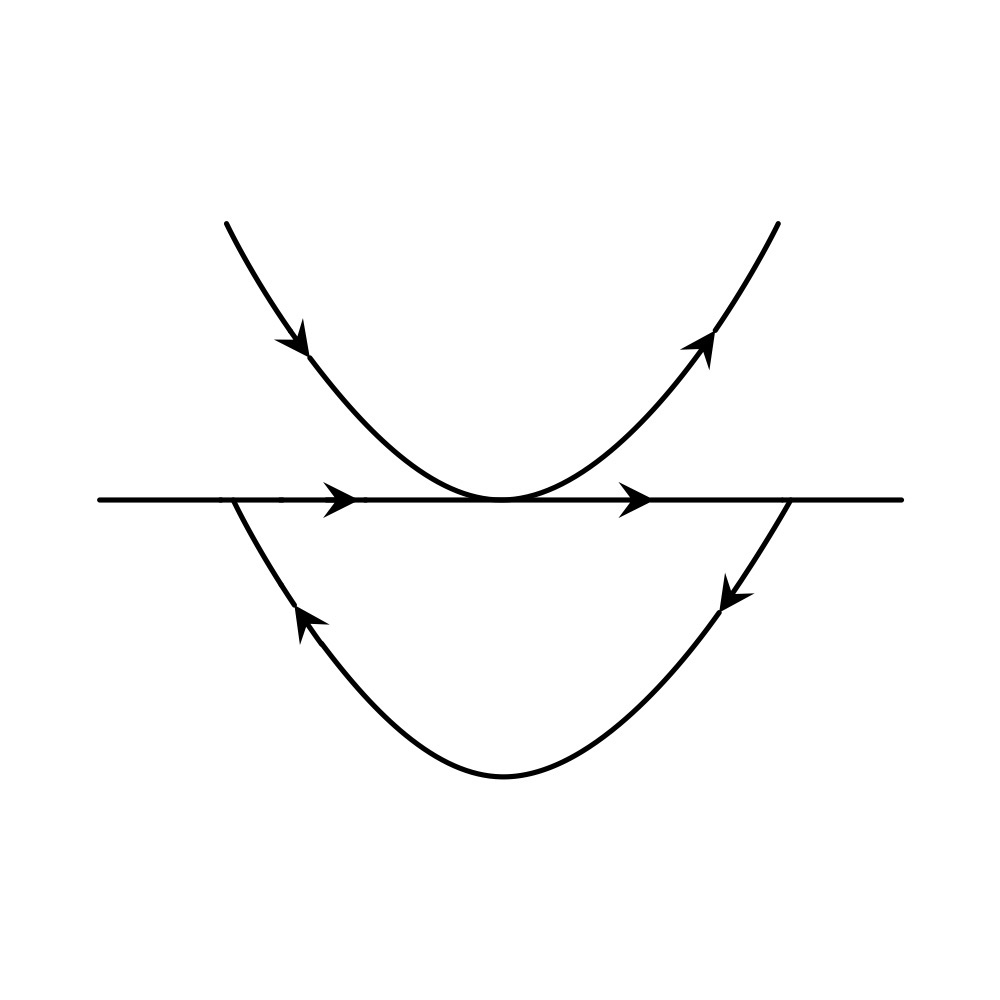}
			\caption{Index equal to $1$}
			\label{fig:dd5}
		\end{subfigure}\\
		\begin{subfigure}{.25\textwidth}
			\includegraphics[width=.95\textwidth]{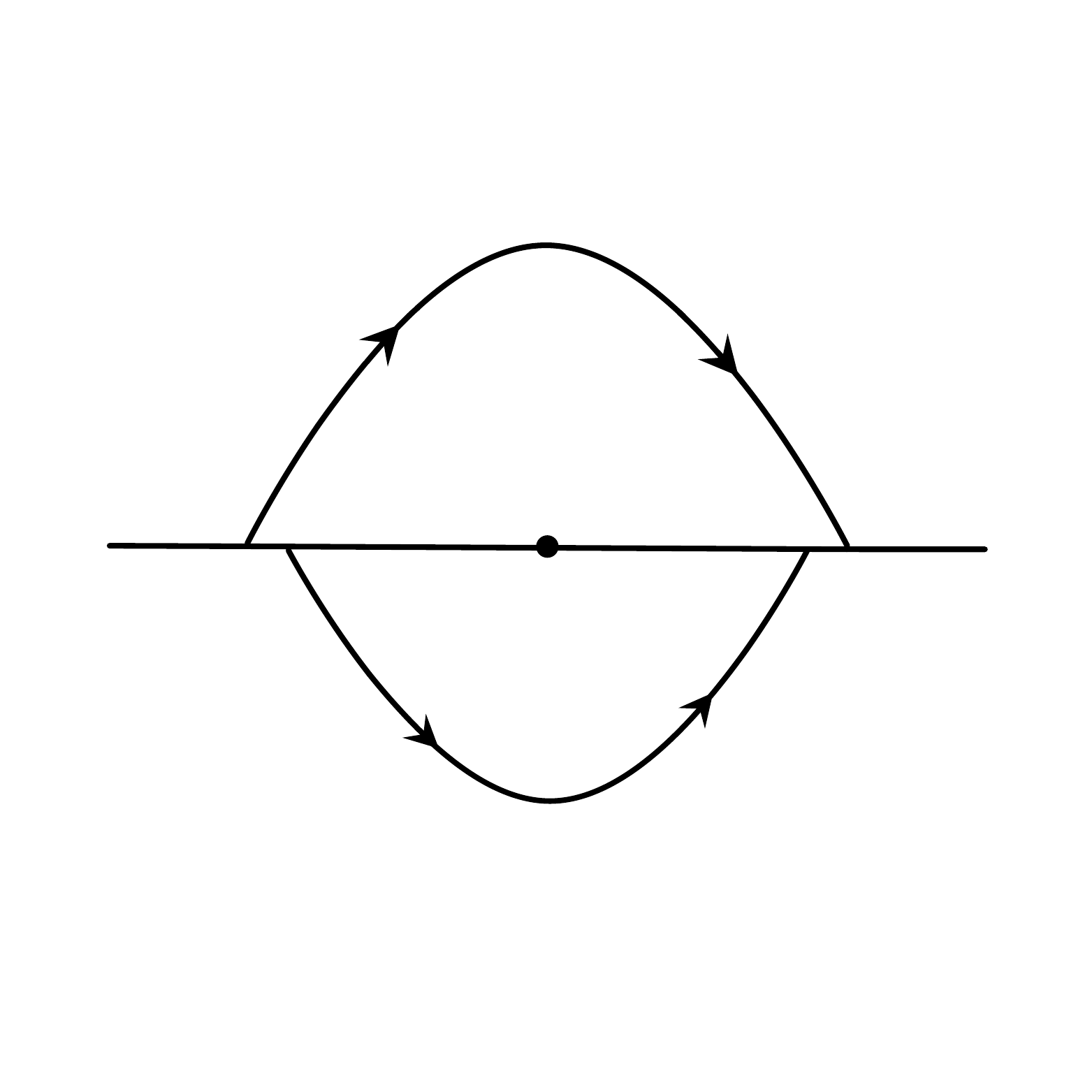}
			\caption{Index equal to $0$}
			\label{fig:dd6}
		\end{subfigure}%
		\begin{subfigure}{.25\textwidth}
			\includegraphics[width=.95\textwidth]{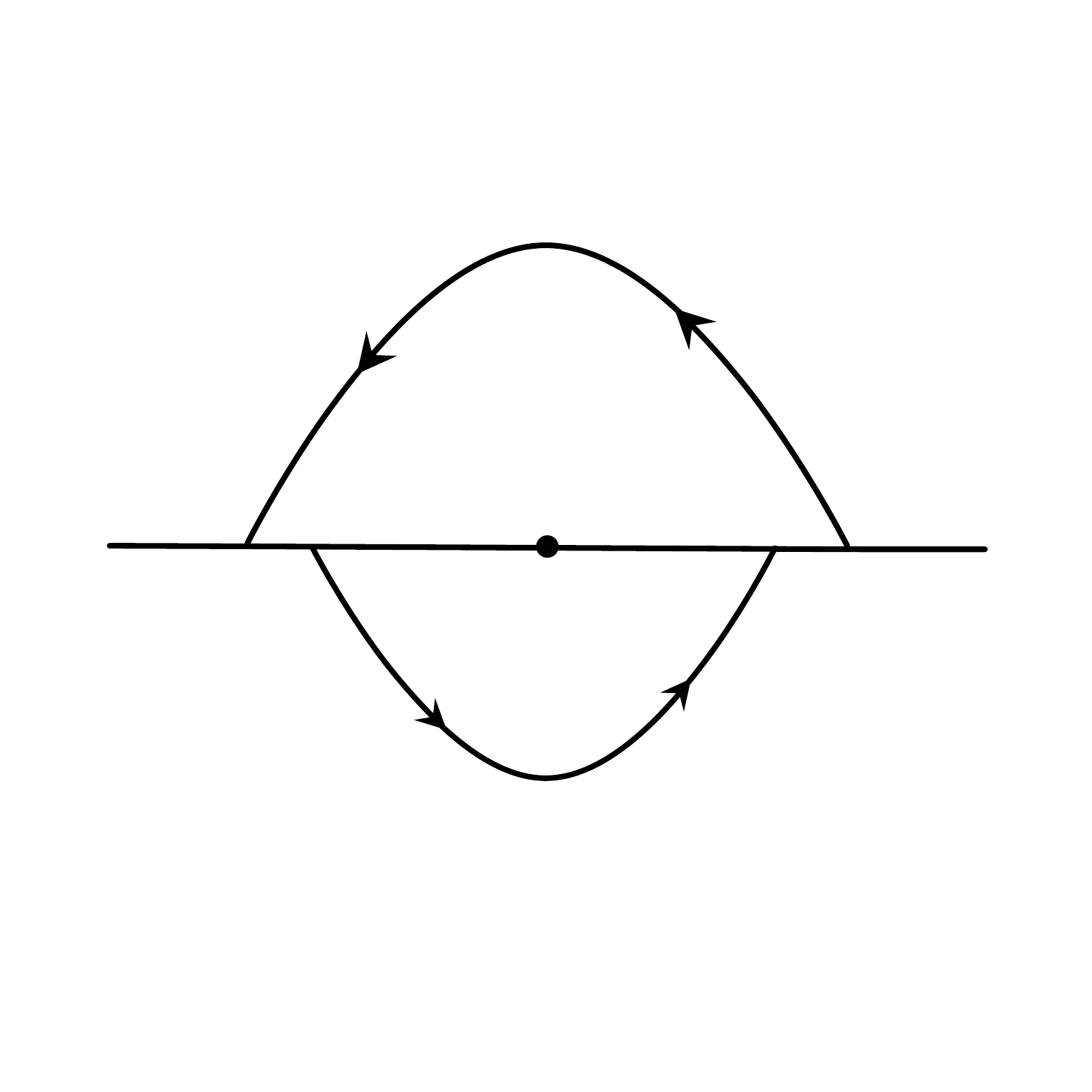}
			\caption{Index equal to $1$}
			\label{fig:dd7}
		\end{subfigure}%
		\caption{Indices for all the possible topological configurations of fold-fold tangential singularities.}
		\label{fig:dd}
	\end{figure}

\begin{proposition}\label{Prop.foldfold} Let  $p\in\Sigma$ be a fold-fold tangential singularity of a Filippov vector field $Z=(F^+,F^-)$ and assume that $k\neq0$.	Then, the following statements hold: 
	\begin{itemize}
		\item[(a)] If $F^+_1(0,0)F^-_1(0,0)>0$, then $I_p(Z)=0.$
		\item[(b)] If $F^+_1(0,0)F^-_1(0,0)<0$ and $k>0$, then $I_p(Z)=1.$
		\item[(c)] If $F^+_1(0,0)F^-_1(0,0)<0$ and $k<0$, then $I_p(Z)=-1.$
	\end{itemize}
\end{proposition}
\begin{proof}
	The canonical form of the fold-fold tangential singularity, which is obtained after rescaling the time variable (see \cite{Novaes2021}), is given by 
	\begin{align}\label{foldfoldsimplified}
		Z(x,y) = \left\{\begin{array}{lll}
			F^+(x,y)= (\mu^+, a^+ x + x^2 f^+(x) + y g^+(x, y)),\,\text{if}\,\, y>0,\\
			F^-(x,y)= (\mu^-, a^- x + x^2 f^-(x) + y g^-(x, y)),\,\text{if}\,\, y<0,
		\end{array}\right.
	\end{align}
	where   $$\de^\pm = \sgn(F_1^\pm(0,0)) \quad  \text{ and } \quad a^\pm = \dfrac{1}{|F_1^\pm(0,0)|}\dfrac{\partial F_2^\pm}{\partial x} (0,0).$$ Notice that $p=(0,0)$ is a fold-fold tangential singularity of \eqref{foldfoldsimplified}. 
	
Consider the following 1-parameter family of Filippov vector fields
	\begin{align}\label{foldfoldsimplifiedpertur}
		Z_{\delta}(x,y) = \left\{\begin{array}{l}
			F^+(x,y), \,\,\text{if}\,\, y>0,
			\\
			F^-_{\delta}(x,y),\,\,\text{if}\,\, y<0, 
			\end{array}\right.
	\end{align}
	where $F^-_{\delta}(x,y)=F^-(x-\delta,y)$. Since, for $\delta=0$, $Z_0(x,y)=Z(x,y)$, Proposition \ref{prop:invperd} provides the existence of $\bar{\delta}>0$ such that, for each $\delta\in (0,\bar\delta)$, $Z_{\delta}$ has no singularities on $\p B$  and 
\begin{equation}\label{index-fold-fold}
I_{\partial B}(Z)=I_{\partial B}(Z_{\delta}).
\end{equation}
 In what follows, we will apply Propostions \ref{prop:somaindexsingu} and \ref{prophiperbolic}  to compute the index $I_{\partial B}(Z_{\delta})$.
	
First, we shall determine the crossing, sliding, and escaping regions of $Z_{\delta}$. Notice that the points $(0,0)$ and $(\delta,0)$ correspond to regular-fold tangential singularities of $Z_{\delta}$, which are contact points of $F^+$ and $F^-_{\delta}$, respectively. Moreover,
\[
(F^+f)( x,0)\cdot(F^-_{\delta} f)( x,0) =x(x-\delta)\big( a^- a^+ +O(x)+ O(\delta )\big)
\]
and
\[
\dfrac{(F^+f)(x, 0)}{x}= a^+ +\CO(x),
\]
where $f(x,y)=y$. Since $a^- a^+\neq0$ and $x\in[-r,r]$, we can take $r>0$ and $0<\delta_0<r$ smaller in order that
 \[
 \sgn\Big((F^+ f)( x,0)\cdot(F^-_{\delta}  f)( x,0)\Big)=\sgn\big(a^- a^+ x(x-\delta)\big)
 \]
 and
 \[
\sgn\Big((F^+f)(x, 0)\Big)= \sgn(a^+ x),
 \]
 for every $\delta\in(0,\delta_0)$ and $x\in[-r,r]$. This allows to distinguish the crossing, sliding, and escaping regions of $Z_{\delta}$ for $x\in[-r,r]$, which is summarized in \autoref{tab:regions-foldfold}.
	\begin{table}[h]
		\centering
		\begin{tabular}{|c|c|c|c|c|}
			\hline			
			$a^+$ & $a^-$ & $-r\leq x < 0$  & $0 < x < \delta $ & $\delta<x\leq r$  \\\hline\hline
			$+$   & $+$   & crossing & escaping    & crossing \\\hline
			$-$   & $-$   & crossing & sliding     & crossing \\\hline
			$+$   & $-$   & sliding  & crossing    & escaping \\\hline
			$-$   & $+$   & escaping & crossing    & sliding  \\\hline
		\end{tabular}
		\caption{Crossing, sliding, and escaping regions of $Z_{\delta}$ for $x\in[-r,r]$ and $\delta\in(0,\delta_0)$, with $r>0$ and $ 0<\delta_0 < r $ sufficiently small.}
		\label{tab:regions-foldfold}
	\end{table}

Besides the regular-fold tangential singularities $(0,0)$ and $(\delta,0)$, $Z_{\delta}$ may present pseudo-equilibria inside the sliding and escaping regions. Thus, in the sequel, we need to determine the existence of zeros of the sliding vector field $Z^s_{\delta}(x)$ (given by \eqref{slid}) for $x\in[-r,r]$ and $\delta\in(0,\delta_0)$. To do that consider the normalization of the sliding vector field $Z^s_{\delta}(x)$
\[	
		A(\delta, u)=	 (F^- _{\delta}f)(x,0) F^+(x,0)- (F^+ f)(x,0) F^-_{\delta}(x,0),
\]
Let $k$ be given by \ref{k}, that is, $k=a^+\mu^- - a^-\mu^+\neq0$ for \eqref{foldfoldsimplifiedpertur}. 	
Since
	\[
	A(0,0)=0 \quad\text{and}\quad \dfrac{\p A}{\p x}(0,0)=-k\neq0,
	\]
	the Implicit Function Theorem provides the existence of $x^*(\delta),$ defined for $\delta\in[0,\delta_0)$ with $\delta_0$ sufficiently small, satisfying $x^*(0)=0$ and $A(\delta,x^*(\delta))=0$. Moreover, one can see that
	\[
	x^*(\delta)=\delta u^*+\CO(\delta^2),
	\]	
	where
	\begin{equation}\label{u*}
 u^* = -\dfrac{a^-\mu^+}{k} = \dfrac{1}{1-\frac{a^+\mu^-}{a^-\mu^+}}\notin\{0,1\}.
	\end{equation}
	Therefore, for $\delta\in(0,\delta_0)$, $x^*(\delta)$ is the unique possible zero of the the sliding vector field $ Z_{\delta}^s$ bifurcating from the fold-fold tangential singularity of $Z$. In order to determine if it is indeed a pseudo-equilibrium, we have to check whether it is contained or not in $\Sigma^{s,e}$.
	
	First, suppose that $a^+a^->0$. From \autoref{tab:regions-foldfold}, $x^*(\delta)$ is a pseudo-equilibrium provided that  $0 < x^*(\delta) < \delta$ for every $\delta\in(0,\delta_0)$, which is equivalent to $0<u^*<1$. From \eqref{u*}, the last inequality is equivalent to $\mu^+\mu^-=-1$.

Now, assume that  $a^+a^-<0$. From \autoref{tab:regions-foldfold}, $x^*(\delta)$ is a pseudo-equilibrium provided that  $x^*(\delta)<\delta \bar u_0(\delta)$ or $x^*(\delta)> \delta \bar u_1(\delta)$ for every $\delta\in(0,\delta_0)$, which is equivalent to	$u^*<0$ or $u^*>1$. Again, from \eqref{u*}, $u^*<0$ or $u^*>1$ is equivalent to $\mu^+\mu^-=-1$.
	
	 Thus, taking the above comments into account, $x^*(\delta)$ is a pseudo-equilibrium if and only if  $\mu^+\mu^-=-1$. Hence, for $ \mu^+\mu^- = 1 $ there is no pseudo-equilibrium inside $B$ and, therefore, Propostions \ref{prop:somaindexsingu} and \ref{prophiperbolic} implies that $I_{\partial B}(Z_{\delta})=0$ for every $\delta\in(0,\delta_0)$. The proof of statement (a) follows from the relationship \eqref{index-fold-fold}.
	 
	 Now, assume that $ \mu^+\mu^- = -1. $ Notice that
	 \[
	 Z_{\delta}^s=\dfrac{A(\delta,x)}{\Delta_{\delta}(x)},
	 \]
	 where $ \Delta_{\delta}(x)= (F^-_{\delta} f)( x,0)-(F^+ f)( x,0)$ only vanishes for $x\in\{0,\delta\}$. Moreover, $\Delta_{\delta}(x)>0$ (resp. $\Delta_{\delta}(x)<0$) provided that $(x,0)$ belongs to the sliding (resp. escaping) region of $Z_{\delta}$. Since,
\[
 \dfrac{\p A}{\p x}(\delta,x^*(\delta))=-k+\CO(\delta),
 \]
	one has that 
	 \[
	 \sgn\left(\dfrac{\p Z_{\delta}^s}{\p x}(\delta,x^*(\delta))\right)=-k\,\sgn\big(\Delta_{\delta}(x^*(\delta))\big)\neq0.
	 \]
It implies that $x^*(\delta)$ is indeed a hyperbolic pseudo-equilibrium of $Z_{\delta}$, which is pseudo-node for $k>0$ and a pseudo-saddle for $k<0$. Hence Propostions \ref{prop:somaindexsingu} and \ref{prophiperbolic} implies that, for every $\delta\in(0,\delta_0)$, $I_{\partial B}(Z_{\delta})=1$ for $k>0$ and $I_{\partial B}(Z_{\delta})=1$ for $k<0$. The proof of statements (b) and (c) follow from the relationship \eqref{index-fold-fold}.
	 \end{proof}	

\subsection{Indices of regular-cusp tangential singularities and saddle-node pseudo-equilibria}

\begin{proposition}\label{rcsaps}
If  $p\in\Sigma$ is a regular-cusp tangential singulairty or a saddle-node pseudo-equilibrium of the Filippov vector field $Z=(F^+,F^-),$ then 
$I_p(Z)=0.$
\end{proposition}

\begin{proof}
Let $p$ a saddle-node pseudo-equilibrium of the Filippov vector field $Z,$ that is, $p$ is a saddle-node of the sliding vector field $Z^s$. Take $r>0$ such that $p$ is the unique singularity of $Z$ in $B=B_r(p).$ It is easy to see that there exists a $1$-parameter family $Z_{\delta}$, perturbation of $Z$, such that the sliding vector field $Z_{\delta}^s$ undergoes a saddle-node bifurcation for $\delta=0$. This means that $Z_{\delta}^s$ does not have any singularities on $B$ for $\delta>0$  (or $\delta<0$ depending on the direction of the bifurcation) sufficiently small. Hence, from Proposition \ref{prop:invperb} we get that  $I_{\partial B}(Z_{\delta})=0$. Finally, Proposition \ref{prop:invperd}  implies that $I_{\partial B}(Z)=I_{\partial B}(Z_{\delta})=0$.

As in the saddle-node pseudo-equilibrium, the regular-cusp tangential singularity can also be removed by means of perturbations. Then, the proof that its index is zero follows in an analogous way.\end{proof}

\section{Poincar\'{e}--Hopf Theorem for Filippov vector fields}\label{sec:PHTthmFVF}

In this section, we state and prove the main result of this paper, related to the Poincaré--Hopf Theorem (the classical version, for smooth vector fields, in  the Appendix).

We demonstrate that, by extending the index for singularities in Filippov vector fields as provided in Section \ref{sec:def}, the Poincaré--Hopf Theorem remains valid for Filippov vector fields defined on 2-dimensional compact manifolds.

Originally proven by Poincaré in two dimensions and generalized by Hopf (see, for instance, \cite{Hopf27} and \cite{poincare}), the Poincaré--Hopf Theorem is one of the most intriguing results in differential topology. It connects global invariants of a manifold, such as the Euler characteristic, with local invariants of vector fields on the manifold, like the indices of singularities (see Theorem \ref{PHt} from the Appendix). This relationship provides a profound link between the topology of a manifold and the behavior of vector fields defined on it.

Let's briefly discuss some concepts related to the Euler characteristic, as we have already talked extensively about the index of singularities. The concept of the Euler characteristic originated with Leonhard Euler in 1758 in the context of convex polyhedra. For a polyhedron $ P $, the Euler characteristic is defined as $ \chi(P) = v - e + f $, where $ v $ is the number of vertices, $ e $ is the number of edges, and $ f $ is the number of faces. This formula, now known as Euler's polyhedron formula, was later extended to higher-dimensional manifolds.

For surfaces, the Euler characteristic is related to another concept, which is the genus of the surface. The genus $g$ of a surface $M$ refers to the number of ``holes'' or ``handles'' in the surface (for a precise definition, we refer to \cite{hatcher,MR0226651}). For example, a sphere has a genus of 0 and no holes, a torus has a genus of 1 and one hole and a surface with two holes, like a double torus, has a genus of 2. The relationship between the Euler characteristic and the genus is given by the formula $\chi(M) = 2 - 2g,$ where $g$ is the genus of $M$.

For a finite-dimensional compact manifold, the Euler characteristic can be indeed computed using a triangulation that transforms the manifold into a polyhedral complex, allowing the use of Euler's classical formula, despite the existence of more efficient constructions for calculation in modern times, such as the Betti numbers' theory or using homology (see, for instance, \cite{MR0226651,hatcher}). The Euler characteristic serves as a very good topological and homotopy invariant for manifolds, in the sense that homeomorphic manifolds, which can be continuously deformed into each other, share the same Euler characteristic (see, for instance, \cite{hatcher}).

Finally, let's state the precise formulation of the main result that we will prove.

\begin{theorem}\label{teoremapoincare}
Let $\CZ$ be the Filippov vector field (defined on a $2$-dimensional compact manifold $M$)  given by \eqref{dds}. Denote the set of the singularities of $\CZ$ by $\CS$ and assume that they are all isolated. Then,
\[
\sum_{p\in\CS} I_p(\CZ)=\chi(M),
\]
where $\chi(M)$ is the Euler Characteristic of $M$.
\end{theorem}
\begin{proof}
Let $p_i\in M$, $i\in\{1, \ldots, n\},$ be the singularities of the Filippov vector field $\CZ$. Consider an atlas $\mathcal{A}=\{(U_{\alpha},\Phi_{\alpha}):\alpha\}$ of $M$ satisfying that:
\begin{enumerate}
\item for each $i\in\{1, \ldots, n\},$  there exists $\alpha_i$ such that $p_i$ is the unique singularity of $\CZ$ inside $U_{\alpha_i}$, and
\item $U_{\alpha_i}\cap U_{\alpha_j}=\emptyset$, for $i\neq j$.
\end{enumerate}
Denote $U^i=U_{\alpha_i}$ and $\Phi^i=\Phi_{\alpha_i}$.  From Definition \ref{indiceM}, $I_{p_i}(\CZ)=I_{\Phi^i(p_i)}(\Phi^i_*\CZ)$. Take $r_i>0,$ for $i\in\{1,\ldots, n\}$, such that $B_i=B_{r_i}(\Phi^i(p_i))\subset D^i:=\Phi^i(U^i)$. Thus, from Definition \ref{indexp}, we have that  $I_{p_i}(\CZ)=I_{\p B_i}(\Phi^i_*\CZ)$.

Now, let $\CZ_{\e}$ be a global ST-regularization of $ \CZ$. Taking Remark \ref{globalreg} into account, the atlas $\A$ can be chosen in such a way that $\Phi^i_*\CZ_\e:D^i\to\mathbb{R}^2$ writes like \eqref{regula}. From Theorem \ref{thm:invreg}, there exists $\bar\e>0$ such that $I_{\p B_i}(\Phi^i_*\CZ)=I_{\p B_i}(\Phi^i_*\CZ_{\e^*})$, for $i\in\{1,\ldots, n\}$ and $\e\in(0,\bar\e]$. So far, we have obtained that
\begin{equation}\label{sum0}
I_{p_i}(\CZ)=I_{\p B_i}(\Phi^i_*\CZ_{\e}),
\end{equation}
for $i\in\{1,\ldots, n\}$ and $\e\in(0,\bar\e]$.

Now, fix $\e^*\in(0,\bar \e]$. Recall that the set $\mathcal{G}\subset C^l(M)$ of the vector fields defined on $M$ having only hyperbolic singularities is open and dense in $C^l(M)$ (see \cite[Theorem 3.4]{palis2012geometric}). Since $\CZ_{\e^*}\in C^l(M),$ there exists a perturbation $\X\in \mathcal{G}$ of $\CZ_{\e^*}$ (as close  $\CZ_{\e^*}$ as we want) satisfying that $\X$ has finitely many singularities and none of them are contained in $M\setminus \bigcup_{i=1}^n (\Phi^i)^{-1}( B_i)$. Thus, for each $i\in\{1,\ldots, n\}$, let $p_i^j,$ $j\in\{1,\ldots, m_i\}$, be the singularities of $\X$ in $(\Phi^i)^{-1}( B_i)$. Hence,
\begin{equation}\label{sum2}
\sum_{j=1}^{m_i} I_{p_i^j}(\X)=\sum_{j=1}^{m_i} I_{\Phi^i(p_i^j)}(\Phi^i_*\X)=I_{\p B_i}(\Phi^i_*\X).
\end{equation}
Notice that $\X$ can be  taken sufficiently close to $\CZ$ in order that $||\Phi^i_*\X(\x)-\Phi^i_* \CZ_{\e^*}(\x)||<||\Phi^i_* \CZ_{\e^*}(\x)||$ for every $\x\in\p B_i$ and $i\in\{1,\ldots,n\}$. Thus, Proposition \ref{prop:pertsmooth} from the Appendix implies that
\[
I_{\p B_i}(\Phi^i_*\X)=I_{\p B_i}(\Phi^i_* \CZ_{\e^*}),
\]
which, together with the relationship \eqref{sum0} and \eqref{sum2}, provides
\begin{equation}\label{sum1}
\sum_{j=1}^{m_i} I_{p_i^j}(\X)=I_{p_i}(\CZ).
\end{equation}

Finally, applying the Poincar\'{e}-Hopf Theorem for the smooth vector field $\X$ (see Theorem \ref{PHt} from the Appendix), it follows that
\begin{equation}\label{sum4}
\sum_{i=1}^n\sum_{j=1}^{m_i} I_{p_i^j}(\X)=\chi(M).
\end{equation}
Therefore, from \eqref{sum1} and \eqref{sum4}, we conclude that
\[
\sum_{i=1}^n I_{p_i}(Z)=\sum_{i=1}^n\sum_{j=1}^{m_i} I_{p_i^j}(\X)=\chi(M).
\]
\end{proof}

Taking into account that the Euler Characteristic of a sphere is $2$, we obtain, as a direct consequence of Theorem \ref{teoremapoincare}, the following version of the Hairy Ball Theorem for Filippov vector fields.
\begin{corollary}\label{hairy}
Assume that $M$ is a smooth sphere and let $\CZ$ be the Filippov vector field given by \eqref{dds} defined on $M$. Then, $\CZ$ has at least one singularity with positive index.
\end{corollary}

\section{Invariance under regularization process: proof of Theorem \ref{thm:invreg}}\label{sectiom:invariaceunderreg}
This section is devoted to the proof of Theorem \ref{thm:invreg}. Consider the Filippov vector field $Z$ given by \eqref{locdds} and let $Z_{\e}$ be its ST-regularization given by \eqref{regula}. Notice that  $Z_{\e}(\x)=(X_{\e}(\x),Y_{\e}(\x))$, where
\begin{equation}\label{Ze}
\begin{array}{lll}
X_{\e}(\x)&=&\dfrac{F^+_1(\x)+F^-_1(\x)+\phi_{\e}(y)(F^+_1(\x)-F^-_1(\x))}{2},\vspace{0.3cm}\\
Y_{\e}(\x)&=&\dfrac{F^+_2(\x)+F^-_2(\x)+\phi_{\e}(y)(F^+_2(\x)-F^-_2(\x))}{2}.
\end{array}
\end{equation}

Let $B=B_r(\x_0) \subset D $ and consider the following parametrization of its boundary $\partial B$, $$\sigma(t)=(u(t),v(t))=\big(r\cos(t+\pi/2),r\sin(t+\pi/2)\big),\,\,\, t\in[0,2\pi]$$
Lemma \ref{analucia1} implies that $Z_{\e}$ does not vanish on $\p B$ for $\e>0$ sufficiently small. Since $Z_{\e}$ is a smooth vector field, we know that
\begin{equation}\label{indiceZe}
I_{\partial B}(Z_{\e})= \dfrac{1}{2\pi}\int_{\Gamma_{\e}} \omega_W=\displaystyle\dfrac{1}{2 \pi}\int_{0}^{2\pi} \Big(p_{Z_{\e}}(\sigma(t))+q_{Z_{\e}}(\sigma(t))\Big)\d t,
\end{equation}
where $\Gamma_{\e}=\{Z_{\e}\circ\sigma(t),\, t\in[0,2\pi] \}$ and, for a vector field  $A(\x)$, the functions $p_A$ and $q_A$ are given by expressions in \eqref{pA.qA} of the Appendix.

Since $0$ is a regular value of $f$, one can apply the Implicit Function Theorem to ensure that,  for $\e>0$ sufficiently small, the curve $\partial B$ intersects the boundaries of the regularization band,  $L^+_{\e}=\{(x,y)\in D:\,f(x,y)=\e\}$ and $L^-_{\e}=\{(x,y)\in D:\,f(x,y)=-\e\}$, at points $\sigma(w_i^{\e})$, for $i\in\{1,\ldots 4\}$  { in such a way that 
\begin{equation}\label{relwi}
w_1^{\e},w_2^{\e}\to  \pi/2\,\,\,\text{and}\,\,\, w_3^{\e},w_4^{\e}\to 3\pi/2\,\,\,\text{as}\,\,\, \e\to 0.
\end{equation}}
Also, let $\partial B=\partial B^1_{\e}\cup \partial B^2_{\e}\cup \partial B^3_{\e} \cup \partial B^4_{\e} \cup \partial B^5_{\e}$, where, for each $i\in \{1,2,3,4,5\},$ $\partial B^i_{\e}=\{\sigma(t): \, t \in (w_{i-1}^{\e},w_i^{\e})\}$ (see Figure \ref{fig:rep}). 
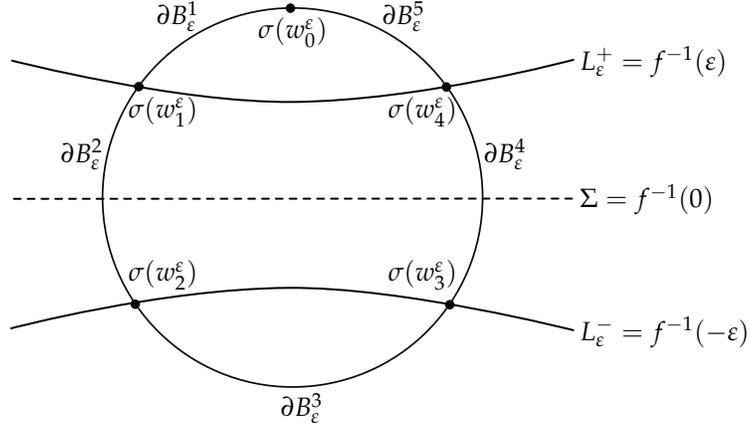
\begin{figure}[H]
\centering 
\begin{overpic}[width=7.5cm]{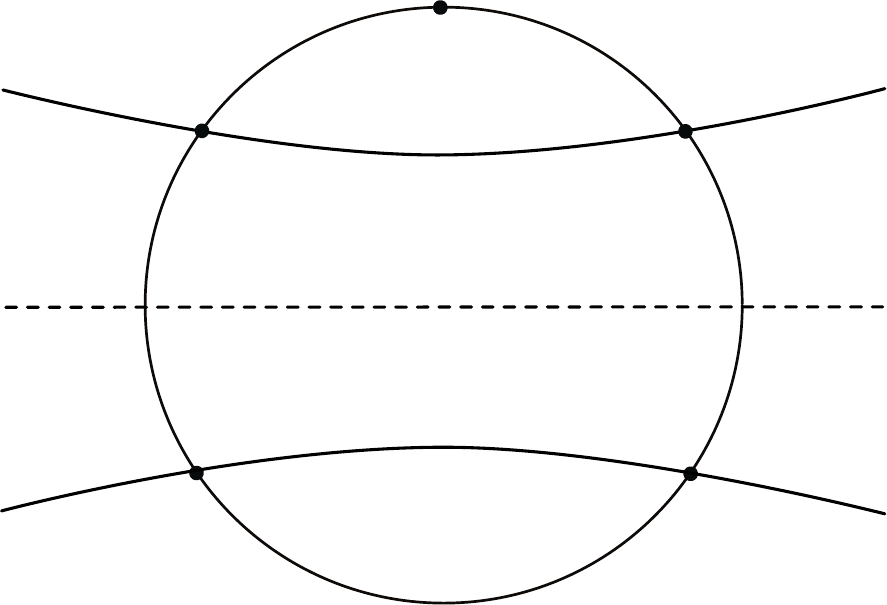}
\put(44,62){$\sigma(w_0^{\e})$}
\put(21,48){$\sigma(w_1^{\e})$}
\put(21,19){$\sigma(w_2^{\e})$}
\put(67,19){$\sigma(w_3^{\e})$}
\put(67,48){$\sigma(w_4^{\e})$}
\put(26,64.5){$\partial B^1_{\e}$}
\put(9,40){$\partial B^2_{\e}$}
\put(48,-5){$\partial B^3_{\e}$}
\put(84,40){$\partial B^4_{\e}$}
\put(66,64.5){$\partial B^5_{\e}$}
\put(101,56.5){$L^+_{\e}=f^{-1}(\e)$}
\put(101,32){$\Sigma=f^{-1}(0)$}
\put(101,8.5){$L^-_{\e}=f^{-1}(-\e)$}
\end{overpic}
\vspace{0.5cm}
\caption{Illustration of the curve $\partial B$ and its intersection with the boundaries of the regularization band, $L^+_{\e}$ and $L^-_{\e}$.}\label{fig:rep}
     \end{figure}

The index \eqref{indiceZe} can be split in several integrals as follows
\begin{equation}
\begin{array}{lll}\label{indicesuave}\displaystyle
    I_{\partial B}(Z_{\e})= \dfrac{1}{2\pi}\left(\int_{\Gamma_{\e}^1} \omega_W+\int_{\Gamma_{\e}^2} \omega_W+\int_{\Gamma_{\e}^3} \omega_W+\int_{\Gamma_{\e}^4} \omega_W+\int_{\Gamma_{\e}^5} \omega_W\right),
\end{array}
\end{equation}
where $\Gamma_{\e}^i=\{Z_{\e}(\x),\, \x\in \partial B^i_{\e}\}$ for $i\in\{1,2,3,4,5\}.$ Notice that, for $\x \in \Gamma_{\e}^1 \cup \Gamma_{\e}^5, $ $Z_{\e}(\x)=F^+(\x)$; for $\x\in \Gamma_{\e}^3,$ $Z_{\e}(\x)=F^-(\x);$ and for $\x\in \Gamma_{\e}^2 \cup \Gamma_{\e}^4, $ $Z_{\e}(\x)=(X_{\e}(\x),Y_{\e}(\x))$ is given by \eqref{Ze}.

In order to prove that $I_{\partial B}(Z)=I_{\partial B}(Z_{\e})$ for $\e>0$ small enough, it is sufficient to show that 
\begin{equation*}\label{goallimit}
\lim\limits_{\e \to 0}I_{\partial B}(Z_{\e})=I_{\partial B}(Z),
\end{equation*} because, since the index is a discrete function, we know that $\lim\limits_{\e \to 0}I_{\partial B}(Z_{\e})=I_{\partial B}(Z_{\e})$ for $\e>0$ sufficiently small.

{First, along  $\Gamma_{\e}^1$,  $\Gamma_{\e}^3$, and  $\Gamma_{\e}^5$, we are outside the regularization band, where the integrand does not depend on $\e$. In this case, taking into account the limits in \eqref{relwi}, we get}
\[
\begin{aligned}
\displaystyle \lim\limits_{\e \to 0}\displaystyle\int_{\Gamma_{\e}^1} \omega_W=&\displaystyle \lim\limits_{\e \to 0}\displaystyle \int_{w_0^{\e}}^{w_1^{\e}}\Big(p_{F^+}(\sigma(t)) +q_{F^+}(\sigma(t)) \Big)\d t=\displaystyle\int_{0}^{\pi/2}\Big(p_{F^+}(\sigma(t)) +q_{F^+}(\sigma(t)) \Big)\d t,\vspace{0.3cm}\\
\displaystyle \lim\limits_{\e \to 0}\displaystyle\int_{\Gamma_{\e}^3} \omega_W=&\displaystyle \lim\limits_{\e \to 0}\displaystyle \int_{w_2^{\e}}^{w_3^{\e}}\Big(p_{F^-}(\sigma(t))+q_{F^-}(\sigma(t)) \Big)\d t
=\displaystyle\int_{\pi/2}^{3\pi/2}\Big(p_{F^-}(\sigma(t)) +q_{F^-}(\sigma(t))\Big)\d t,\vspace{0.3cm}\\
\displaystyle \lim\limits_{\e \to 0}\displaystyle\int_{\Gamma_{\e}^5} \omega_W=&\displaystyle \lim\limits_{\e \to 0}\displaystyle \int_{w_4^{\e}}^{w_5^{\e}}\Big(p_{F^+}(\sigma(t)) +q_{F^+}(\sigma(t)) \Big)\d t
=\displaystyle\int_{3\pi/2}^{2\pi}\Big(p_{F^+}(\sigma(t)) +q_{F^+}(\sigma(t)) \Big)\d t,
\end{aligned}
\]
Therefore, adding up the integrals above, we obtain
\[
\displaystyle \lim\limits_{\e \to 0}\displaystyle\left(\int_{\Gamma_{\e}^1} \omega_W+\int_{\Gamma_{\e}^3} \omega_W+\int_{\Gamma_{\e}^5} \omega_W\right)= \int_{\Gamma^+} \omega_W +\int_{\Gamma^-} \omega_W,
\]
where $\Gamma^{\pm}=\{Z(\x),\, \x\in \partial B\cap\Sigma^{\pm}\}.$

Now, for the integral along $\Gamma_{\e}^2$ in \eqref{indicesuave}, we proceed with the following change of integration variable  
\[
t=\tau_1^{\e}(s)=(1-s) w_1^{\e}+s\,w_2^{\e},
\]
which provides
\begin{equation}
\begin{array}{lll}\label{integrala}
\displaystyle\int_{\Gamma^2_{\e}} \omega_W&=&\displaystyle \int_{w_1^{\e}}^{w_2^{\e}}\Big(p_{Z_{\e}}(\sigma(t))+q_{Z_{\e}}(\sigma(t))\Big)dt\vspace{0.2cm}\\
&=&\displaystyle \int_{0}^{1}\Big(p_{Z_{\e}}(\sigma\circ\tau_1^{\e}(s)) +q_{Z_{\e}}(\sigma\circ\tau_1^{\e}(s))\Big)\dfrac{d \tau_1^{\e}}{ds}(s)ds.
\end{array}
\end{equation}
Notice that the integrand of the equation \eqref{integrala} is uniformly convergent as $\e$ goes to $0.$ Then, by switching the limit with the integral and { taking into account the limits in \eqref{relwi}}, we get
\begin{equation}\label{integrala2}
\begin{array}{lll}
\lim\limits_{\e \to 0}\displaystyle\int_{\Gamma_{\e}^2}\omega_W&=&\displaystyle \int_{0}^{1}\lim\limits_{\e \to 0}\Big(p_{Z_{\e}}(\sigma\circ\tau_1^{\e}(s)) +q_{Z_{\e}}(\sigma\circ\tau_1^{\e}(s))\Big)\dfrac{d \tau_1^{\e}}{ds}(s)ds\vspace{0.2cm}\\
&=&\displaystyle \int_{0}^{1}\dfrac{4 \phi '(1-2 s) (F^+_1(-r,0) F^-_2(-r,0)-F^+_2(-r,0) F^-_1(-r,0))}{G_1(s)} ds,
\end{array}
\end{equation}
where 
\[
\begin{array}{lll}
G_1(s)&=&-2 \left(\phi (1-2 s)^2-1\right) F^+_1(-r,0) F^-_1(-r,0)+(\phi (1-2 s)+1)^2 F^+_1(-r,0)^2\\
& &-2 \left(\phi (1-2 s)^2-1\right) F^+_2(-r,0) F^-_2(-r,0)+(\phi (1-2 s)+1)^2 F^+_2(-r,0)^2\\
& &+(\phi (1-2 s)-1)^2 \left(F^-_1(-r,0)^2+F^-_2(-r,0)^2\right).
\end{array}
\]
By taking $u=\phi (1-2 s)-1$, i.e. $s=\sigma_1(u) :=(1-\phi^{-1}(u + 1))/2$, the integral \eqref{integrala2} becomes
\begin{equation}\label{segundamudanca}
\lim\limits_{\e \to 0}\displaystyle\int_{\Gamma_{\e}^2}\omega_W=\displaystyle\int_{0}^{-2} \dfrac{\alpha_1}{\beta_1 (u+2)^2-2 \psi_1 (u+2) u+\eta_1 u^2} d u,
\end{equation}
where $\alpha_1=\alpha(-r,0), ~~ \beta_1=\beta(-r,0), ~~ \psi_1=\psi(-r,0),$ and $\eta_1=\eta(-r,0)$ with
\begin{equation}
\begin{array}{lll}\label{sub1}
\alpha(\x)&:=& 2 F^+_2(\x) F^-_1\x-2 F^+_1(\x) F^-_2(\x),\\
\beta(\x)&:=& F^+_1(\x)^2+F^+_2(\x)^2,\\
\psi(\x)&:=& F^+_1(\x) F^-_1(\x)+F^+_2(\x) F^-_2(\x),\\
\eta(\x)&:=& F^-_1(\x)^2+F^-_2(\x)^2.
\end{array}
\end{equation}
Notice that the discriminant of the denominator of \eqref{segundamudanca}, given by $\psi_1^2 - \beta_1 \eta_1$, is less than or equal to zero. One can see that the discriminant is zero if, and only if, $\det ( F^+(- r,0) | F^-(- r,0) ) = 0$. In this case, $\alpha_1 = 0$ which implies that the integral is zero. Now assume that $\psi_1^2 - \beta_1 \eta_1<0. $ Hence,
\begin{equation}\label{integralgamma2}
\begin{array}{lll}
\lim\limits_{\e \to 0}\displaystyle\int_{\Gamma_{\e}^2}\omega_W&=&\displaystyle\int_{0}^{-2} \dfrac{\alpha_1}{\beta_1 (u+2)^2-2 \psi_1 (u+2) u+\eta_1 u^2} d u\vspace{0.3cm}\\
&=& \dfrac{\alpha_1 \left(\tan^{-1}\left(\frac{\psi_1-\beta_1}{\sqrt{\beta_1 \eta_1-\psi_1^2}}\right)+\tan^{-1}\left(\frac{\psi_1-\eta_1}{\sqrt{\beta_1 \eta_1-\psi_1^2}}\right)\right)}{2 \sqrt{\beta_1 \eta_1-\psi_1^2}}.
\end{array}
\end{equation}
Substituting \eqref{sub1} into \eqref{integralgamma2}, we get
\begin{equation}\label{Gamma2antesH}
\begin{aligned}
\lim\limits_{\e \to 0}\displaystyle\int_{\Gamma_{\e}^2}\omega_W=&\tan ^{-1}\left(\dfrac{F^-_1(-r,0)^2+F^-_2(-r,0)^2-F^+_1(-r,0) F^-_1(-r,0)-F^+_2(-r,0) F^-_2(-r,0)}{F^+_1(-r,0) F^-_2(-r,0)-F^+_2(-r,0) F^-_1(-r,0)}\right)\vspace{0.5cm}\\
+&\tan ^{-1}\left(\dfrac{F^+_1(-r,0)^2+F^+_2(-r,0)^2-F^+_1(-r,0) F^-_1(-r,0)-F^+_2(-r,0) F^-_2(-r,0)}{F^+_1(-r,0) F^-_2(-r,0)-F^+_2(-r,0) F^-_1(-r,0)}\right).
\end{aligned}
\end{equation}
Notice that the argument of the arctangent in \eqref{Gamma2antesH} is of the form
\begin{equation}\label{Hantesleicossenos}
\dfrac{F_1(\x)^2+F_2(\x)^2-F_1(\x)G_1\x-F_2(\x)G_2(\x)}{G_1(\x)F_2\x-G_2(\x)F_1(\x)},
\end{equation}
where $F(\x)=(F_1(\x),F_2(\x))$ and $G(\x)=(G_1(\x),G_2(\x))$ are general vector fields. Notice that \eqref{Hantesleicossenos} becomes $H_{(F,G)}(\x)$, defined in \eqref{HFG}, that is
\[
\dfrac{\|F(\x)\|^2 - \langle F(\x),G(\x)\rangle}{
\det \big( G(\x) | F(\x) \big) } = H_{(F,G)}(\x).
\]
Then,
\[
\lim\limits_{\e \to 0}\displaystyle\int_{\Gamma_{\e}^2}\omega_W=\tan^{-1}\left(H_{(F^-,F^+)}(-r,0)\right)-\tan^{-1}\left(H_{(F^+,F^-)}(-r,0)\right).
\]
Hence, we have obtained that
\[
\lim\limits_{\e \to 0}\displaystyle\int_{\Gamma_{\e}^2}\omega_W=\begin{cases}\tan^{-1}\left(H_{(F^-,F^+)}(-r,0)\right)-\tan^{-1}\left(H_{(F^+,F^-)}(-r,0)\right),& \det ( F^+(- r,0) | F^-(- r,0) )\neq0,\\
0,& \det ( F^+(- r,0) | F^-(- r,0) )=0,
\end{cases}
\]
which coincides with $J_{\partial B}^-(Z)$.

Analogously, for the integral along $\Gamma_{\e}^4$ in \eqref{indicesuave} we proceed with the following change of integration variable  
\[
\tau_2^{\e}(s)=(1-s) w_4^{\e}+sw_3^{\e},
\]
which provides
\begin{equation}\label{integralb}
\begin{array}{lll}
\displaystyle\int_{\Gamma_{\e}^4} \omega_W&=&\displaystyle \int_{w_3^{\e}}^{w_4^{\e}}\Big(p_{Z_{\e}}(\sigma(t))+q_{Z_{\e}}(\sigma(t))\Big)dt\vspace{0.2cm}\\
&=&\displaystyle \int_{0}^{1}\Big(p_{Z_{\e}}(\sigma\circ\tau_2^{\e}(s)) +q_{Z_{\e}}(\sigma\circ\tau_2^{\e}(s))\Big)\dfrac{d \tau_2^{\e}}{ds}(s)ds,
\end{array}
\end{equation}
Again, as the integrand of the equation \eqref{integralb} is uniformly convergent as $\e$ goes to $0$, by switching the limit with the integral{ and taking into account the limits in \eqref{relwi}}, we get \begin{equation}\label{integralb2}
\begin{array}{lll}
\,\,\, \lim\limits_{\e \to 0}\displaystyle\int_{\Gamma_{\e}^4}\omega_W&=&\displaystyle \int_{0}^{1}\lim\limits_{\e \to 0}\Big(p_{Z_{\e}}(\sigma\circ\tau_2^{\e}(s)) +q_{Z_{\e}}(\sigma\circ\tau_2^{\e}(s))\Big)\dfrac{d \tau_2^{\e}}{ds}(s)ds\vspace{0.2cm}\\
&=& \displaystyle\int_{0}^{1} \dfrac{4 \phi '(2 s-1) (F^+_2(r,0) F^-_1(r,0)-F^+_1(r,0) F^-_2(r,0))}{G_2(s)} d s,
\end{array}
\end{equation}
where 
\[
\begin{array}{lll}
G_2(s)&=&\phi (2 s-1)^2 \left((F^+_1(r,0)-F^-_1(r,0))^2+(F^+_2(r,0)-F^-_2(r,0))^2\right)\\& &+2 \phi (2 s-1) \left(F^+_1(r,0)^2+F^+_2(r,0)^2-F^-_1(r,0)^2-F^-_2(r,0)^2\right)\\& &+(F^+_1(r,0)+F^-_1(r,0))^2+(F^+_2(r,0)+F^-_2(r,0))^2.
\end{array}
\]
By taking $u=-1 + \phi(-1 + 2 s)$, i.e. $s=\sigma_2(u) := (\phi ^{(-1)}(u+1)+1)/2$, the integral \eqref{integralb2} becomes

\begin{equation}\label{segundamudancab}
\lim\limits_{\e \to 0}\displaystyle\int_{\Gamma_{\e}^4}\omega_W=\displaystyle\int_{-2}^{0} \dfrac{\alpha_2}{\beta_2 (u+2)^2-2 \psi_2 (u+2) u+\eta_2 u^2} d u,
\end{equation}
where $\alpha_2=\alpha(r,0), ~~ \beta_2=\beta(r,0), ~~ \psi_2=\psi(r,0)$, and $\eta_2=\eta(r,0)$ are given in \eqref{sub1}. 
Notice that the discriminant of the denominator of \eqref{segundamudancab}, given by $\psi_2^2 - \beta_2 \eta_2,$ is less than or equal to zero. One can see that the discriminant is zero if, and only if, $\det ( F^+( r,0) | F^-(r,0) ) = 0$. In this case, $\alpha_2 = 0$ which implies that the integral is zero. Now, assume that  $\psi_1^2 - \beta_1 \eta_1<0. $ Hence,
\begin{equation}\label{integralgamma4}
\begin{array}{lll}
\lim\limits_{\e \to 0}\displaystyle\int_{\Gamma_{\e}^4}\omega_W&=&\displaystyle\int_{-2}^{0} \dfrac{\alpha_2}{\beta_2 (u+2)^2-2 \psi_2 (u+2) u+\eta_2 u^2} d u\vspace{0.3cm}\\
&=& -\dfrac{\alpha_2 \left(\tan^{-1}\left(\frac{\psi_2-\beta_2}{\sqrt{\beta_2 \eta_2-\psi_2^2}}\right)+\tan^{-1}\left(\frac{\psi_2-\eta_2}{\sqrt{\beta_2 \eta_2-\psi^2}}\right)\right)}{2 \sqrt{\beta_2 \eta_2-\psi_2^2}}.
\end{array}
\end{equation}
Substituting \eqref{sub1} into \eqref{integralgamma4}, we get
\[
\lim\limits_{\e \to 0}\displaystyle\int_{\Gamma^4_{\e}}\omega_W=\tan^{-1}\left(H_{(F^+,F^-)}(r,0)\right)-\tan^{-1}\left(H_{(F^-,F^+)}(r,0)\right).
\]
Accordingly, we have concluded that
\[
\lim\limits_{\e \to 0}\displaystyle\int_{\Gamma^4_{\e}}\omega_W=\begin{cases}\tan^{-1}\left(H_{(F^+,F^-)}(r,0)\right)-\tan^{-1}\left(H_{(F^-,F^+)}(r,0)\right),& \det ( F^+(r,0) | F^-(r,0) )\neq0,\\
0,& \det ( F^+(r,0) | F^-(r,0) )=0,
\end{cases}
\]
which coincides with $J_{\partial B}^{+}(Z)$.

Therefore,  for $\e>0$ small enough, we get that 
\[
I_{\partial B}(Z_{\e})=\lim\limits_{\e \to 0}I_{\partial B}(Z_{\e})=J(Z)+\dfrac{1}{2\pi}\left( \int_{\Gamma^+} \omega_W +\int_{\Gamma^-} \omega_W\right)=I_{\partial B}(Z),
\]
which concludes the proof of Theorem \ref{sectiom:invariaceunderreg}.

\section{Conclusion}

The Poincaré--Hopf Theorem is one of the most intriguing results in differential topology, elegantly linking global invariants of a manifold, such as the Euler characteristic, with local properties of vector fields, like the indices of their singularities. This deep connection highlights how the topology of a manifold influences the dynamics of vector fields defined on it.

In this work, we have extended the classical Poincaré-Hopf Theorem to the setting of Filippov vector fields on 2-dimensional compact manifolds with smooth switching manifolds, as presented in Theorem \ref{teoremapoincare}. Additionally, Corollary \ref{hairy} provides a version of the Hairy Ball Theorem, stating that any Filippov vector field on a sphere with a smooth switching manifold must have at least one singularity (in the sense of Definition \ref{def:sing}) with positive index.

This extension is achieved through the introduction of a novel index definition that encompasses the singularities of Filippov vector fields, including pseudo-equilibria and tangential singularities (see Definition \ref{def:sing}). We began by defining the index for Filippov vector fields on a circle $\p B$, where $B$ denotes the closed ball $B = B_{r}(\x_0) = {\x \in \R^2 : |\x - \x_0| \leq r}$ (see Definition \ref{indexfilippov}). The resulting index formula, as given by \ref{eq:index}, generalizes the classical index by incorporating an additional angle term $J_{\partial B}(Z)$ (see Remark \ref{rem:equiJ} and Proposition \ref{prop:equiJ}). This additional term vanishes for smooth vector fields, thereby showing that our index definition extends the classical concept for smooth vector fields. Such a definition was used to define the index of isolated singularites of a planar Filippov vector fields (see Definition \ref{indexp}), which was subsequently generalized to singularities in Filippov vector fields on 2-dimensional manifolds $M$ (see Definition \ref{indiceM}).

A key attribute of this new index is its invariance under the Sotomayor-Teixeira regularization process (see Theorem \ref{thm:invreg}). This invariance allowed us to establish several classical index properties, namely: the index is an integer (see Proposition \ref{prop:invpera}); it is zero on a closed ball $B$ with no singularities (see Proposition \ref{prop:invperb}); and it is invariant under perturbations (see Propositions \ref{prop:invperc} and \ref{prop:invperd}) and homotopy (see Proposition \ref{prop:invpere}). Moreover, the index of $Z$ on a closed ball $B$ can be expressed as the sum of the indices of isolated singularities within $B$ (see Proposition \ref{prop:somaindexsingu}).

With these properties, we successfully computed the indices of all generic $\Sigma$-singularities (see Proposition \ref{prophiperbolic}) and some codimension-1 $\Sigma$-singularities, including fold-fold tangential singularities, regular-cusp tangential singularities, and saddle-node pseudo-equilibria (see Propositions \ref{Prop.foldfold} and \ref{rcsaps}). 

This extension of the Poincaré-Hopf Theorem and index theory offers a new framework for studying the dynamics of Filippov vector fields on 2-dimensional manifolds with smooth switching manifolds. In future work, we plan to extend this analysis to Filippov vector fields on higher-dimensional compact manifolds (possibly with boundary), where the corresponding index theory is yet to be developed.

\section{Appendix}\label{app}
This appendix provides some concepts and results from the index theory for smooth vector fields. We are following the references  \cite{MR2256001,fulton1995,QTheory}.

Let $\omega$ be a {\it differential $1$-form} defined on $\R^2,$ i.e. $\omega=p(\x)dx+q(\x) dy$ where $p$ and $q$ are smooth real functions on $\R^2.$ Let $\gamma$ be a curve on $\R^2$ with smooth parametrization $\alpha:[a,b]\to \R^2,$ $\alpha(t)=(u(t),v(t)).$ The integral of $\omega$ along $\gamma$ is defined as
\begin{equation}\label{intdef}
\int_{\gamma} \omega:=\int_a^b \Big(p(\alpha(t)) u'(t)+q(\alpha(t)) v'(t)\Big)dt.
\end{equation}
The above definition does not depend on the parametrization $\alpha$. Accordingly, the {\it winding number of $\gamma$ around the origin} $W(\gamma)$ is an integer defined by
\[
W(\gamma)=\dfrac{1}{2\pi}\int_{\gamma}\omega_W,
\]
where $\omega_W$ is the following differential $1$-form
\begin{equation}\label{1form}
\omega_W:=\dfrac{-y}{x^2+y^2}dx+\dfrac{x}{x^2+y^2}dy.
\end{equation}

The next result is a useful tool for computing the winding number.

\begin{proposition}[{\cite[Corollary 3.8]{fulton1995}}]\label{homotopyW}
If two smooth oriented closed curves, $\gamma$ and $\delta$, are homotopic in $\R^2\setminus\{(0,0)\}$, then $W(\gamma)=W(\delta)$.
\end{proposition}

Now, consider a smooth vector field $A(\x)=(A_1(\x),A_2(\x))$ defined on an open subset $D\subset\R^2.$  Let $\gamma\subset D$ be a smooth oriented closed curve. Assume that $A(\x)$ in nonsingular on $\gamma$. When a point $\x$ moves one cycle around $\gamma$ in the counterclockwise direction, the vector $A(\x)$ winds around the origin an integral number of revolutions. The total number of revolutions $I_{\gamma}(A)$ is called the {\it rotation number of the vector field $A(\x)$ around $\gamma$} \cite{QTheory} (or {\it index of $A$ on $\gamma$}). Following the notation above,
\begin{equation}\label{IgA}
I_{\gamma}(A):=W(\Gamma)=\dfrac{1}{2\pi}\int_{\Gamma} \omega_W,
\end{equation}
where  $\Gamma:=\{A(\x),\, \x\in\gamma\}.$ Define the auxiliary functions $p_A,q_A:D\subset\R^2\to\R^2$ such that
\begin{equation}\label{pA.qA}
	\begin{aligned}
		p_A(\alpha):=&\dfrac{-A_{2}(\alpha)}{A^2_{1}(\alpha)+A^2_{2}(\alpha)} \nabla A_{1}(\alpha)\\
		q_A(\alpha):=&\dfrac{A_{1}(\alpha)}{A^2_{1}(\alpha)+A^2_{2}(\alpha)} \nabla A_{2}(\alpha).
	\end{aligned}
\end{equation}
Then, \eqref{IgA} becomes
\begin{equation*}
	\begin{aligned}
		I_{\gamma}(A)
		=&\dfrac{1}{2 \pi}\left(\displaystyle\int_{a}^{b}\dfrac{-A_{2}(\alpha(t))}{A^2_{1}(\alpha(t))+A^2_{2}(\alpha(t))} \nabla A_{1}(\alpha(t))\cdot \alpha'(t)dt\right.\\
		&+\displaystyle\left.\int_{a}^{b}\dfrac{A_{1}(\alpha(t))}{A^2_{1}(\alpha(t))+A^2_{2}(\alpha(t))} \nabla A_{2}(\alpha(t))\cdot \alpha'(t)dt\right)\\
		=&\displaystyle\dfrac{1}{2 \pi}\int_{a}^{b} \big(p_A(\alpha(t))+q_A(\alpha(t))\big)\cdot \alpha'(t)\,dt.
	\end{aligned}
\end{equation*}

The next result provides the invariance of $I_{\gamma}(A)$ under change of coordinates by diffeomorphism. Although very intuitive, we were not able to find any reference for a proof of such a result, thus we shall provide it.
\begin{proposition}\label{inv-change}
Let $A:D\subset\R^2\to\R^2$ be a vector field, $\gamma\subset D$ a smooth oriented closed curve, and $\alpha:D\to D^*$ a diffeomorphism, where $D^*=\alpha(D)\subset \R^2$. Assume that $A$ does not vanish on $\gamma$. Then, 
\[
I_{\alpha(\gamma)}(\alpha_*A)=I_{\gamma}(A),
\]
where $\alpha_*A:D^*\to \R^2$ is the pushforward of $A$ by $\alpha$.
\end{proposition}
\begin{proof}
From \eqref{IgA}, $I_{\gamma}(A)=W(\Gamma)$ and $I_{\alpha(\gamma)}(\alpha_*A)=W(\Gamma^*)$, where $\Gamma=\{A(\x):\,x\in\gamma\}$ and
\[
\Gamma^*:=\{\alpha_*A(\alpha(\x)):\,x\in\gamma\}=\{d\alpha(\x)A(\x):\,x\in\gamma\}.
\]

We claim that $\Gamma^*$ is homotopic in $D\setminus\{(0,0)\}$ either to $\Gamma$ or $-\Gamma$.
In what follows we shall construct such a homotopy. 

Let $\theta:[0,1]\to D$ be a parametrization of $\gamma$. Thus $\Gamma$ and $\Gamma^*$ are parametrized, respectively, by $\Theta(t)=A(\theta(t))$ and $\Theta^*(t)=d\alpha(\theta(t))A(\theta(t))$. We know that the set of invertible $2\times 2$ real matrices has two path-connected components, one of them containing the identity matrix $I_2$ and the other one containing $-I_2$. Thus, since $d\alpha(\theta(0))$ is invertible,  there exists a path $M(s)$ of invertible matrices satisfying $M(0)=[d\alpha(\theta(0))]^{-1}$ and either (a) $M(1)=I_2$ or (b) $M(1)=-I_2$. In both cases, consider the homotopy $H(s,t)=M(s)d\alpha(\theta(s\cdot t))A(\theta(t))$. Notice that $H(0,t)=A(\theta(t))=\Theta(t)$ and, in case (a) $ H(1,t)=d\alpha(\theta(s\cdot t))A(\theta(t))=\Theta^*(t)$; and in case (b) $H(1,t)=-d\alpha(\theta(s\cdot t))A(\theta(t))=-\Theta^*(t)$. Furthermore, since $A$ does not vanish on $\gamma$ and $M(s)d\alpha(\theta(s\cdot t))$ is invertible for every $(s,t)\in[0,1]\times[0,1]$, we conclude that $H(s,t)\neq(0,0)$ for every $(s,t)\in[0,1]\times[0,1]$.

Hence, Proposition \ref{homotopyW} implies that either $W(\Gamma^*)= W(\Gamma)$ or $W(\Gamma^*)= W(-\Gamma).$ Taking into account \eqref{intdef} and \eqref{1form}, we see that $W(-\Gamma)=W(\Gamma).$ Therefore,  $I_{\alpha(\gamma)}(\alpha_*A)=W(\Gamma^*)=W(\Gamma)=I_{\gamma}(A).$ 
\end{proof}

If $\x_0\in D$ is an isolate critical point of $A(\x)$, then there exists $r>0$ such that the closed ball $B=B_r(\x_0)\subset D$  does not contain any critical point other than $x_0$. Accordingly, the {\it index $I_{\x_0}(A)$ of the vector field $A$ at the critical point $\x_0$} is defined as
\[
I_{\x_0} (A):=I_{\p B}(A)=\dfrac{1}{2\pi}\int_{\Gamma} \omega_W,
\]
where now  $\Gamma=\{A(\x),\, \x\in \p B\}.$ It is well known that if $\gamma$ is a closed curve on $D$ enclosing a finite number of isolated singularities, $x_1,x_2,\ldots,x_n,$ then
\[
I_\gamma(A)=\sum_{i=1}^n I_{x_i}(A).
\]

The next results, taken from \cite{MR2256001} and \cite{QTheory}  are classical properties of the index for smooth vector fields that have been used throughout the paper.

\begin{proposition}[{\cite[Lemma 6.19]{MR2256001}}]\label{appendixpropcurvas}
Suppose that $\gamma\subset D$ can be continuously deformed into $\gamma'$ without passing through a singularity. Then $$I_{\gamma}(A)=I_{\gamma'}(A)$$
\end{proposition}

\begin{proposition}[{\cite[Property 1]{QTheory}}]\label{appendixprop1}
Let $C_1,C_2\subset D$ be  two closed connected regions. Suppose that the intersection of the interiors of $C_1$ and $C_2$ is empty and let $C=C_1 \cup C_2.$ Then $$I_{\partial C}(A)=I_{\partial C_1}(A)+I_{\partial C_2}(A).$$
\end{proposition}

\begin{proposition}[{\cite[Property 2]{QTheory}}]\label{index0fornosingularity}
  Assume that $A$ has no singularities in a bounded closed connected region $C\subset D$. Then, $I_{\partial C}(A)=0$
\end{proposition}

\begin{proposition}[{\cite[Theorem 1.3]{QTheory}}]\label{prop:pertsmooth}
Let $C\subset D$ be a closed bounded region. Suppose that $A_0$ and $A_1$ are smooth vector field on $D$ such that $A_0$ is nonsingular on $\partial C$ and $||A_1(\x)-A_0(\x)||<||A_0(\x)||$ for every $\x\in\p C$. Then,
$A_1(\x)$ is nonsingular on $\partial D$ and $I_{\partial D}(A_0)=I_{\partial D}(A_1).$
\end{proposition}

Finally, let $\CF:M\to TM$ be a smooth vector field defined on a two dimensional smooth manifold $M$. Assume that $p_0\in M$ is an isolated singularity of $\CF.$ Let $(U,\Phi)$ be a chart of $M$ around $p_0$. Then, the index of $\CZ$ at $p_0$ is defined as $I_{p_0}(\CF):= I_{\Phi(p_0)}(\Phi_* \CF).$ 
The next result is the famous Poincar\'{e}--Euler Theorem (see \cite{Hopf27}) which relates the indices of the singularities of a vector field $\CF$ defined on a compact manifold $M$ with the Euler characteristic of $M$, $\chi(M)$. Details about the Euler characteristic of a compact manifold can be found in \cite{fulton1995}.

\begin{theorem}[Poincar\'e-Hopf Theorem]\label{PHt}
Let $\CF$ be a smooth vector field defined on a $2$-dimensional compact manifold $M$. Denote the set of the singularities of $\CF$ by $\CS$ and assume that they are all isolated. Then,
\[
\sum_{p\in\CS} I_p(\CF)=\chi(M),
\]
where $\chi(M)$ is the Euler Characteristic of $M$.
\end{theorem}

\section*{Statements and Declarations}

JAC  is partially supported by the grants 2018/25575-0 and 2022/01375-7 of the São Paulo Research Foundation (FAPESP). RMM is partially supported by grants 2021/08031-9, 2018/03338-6 and 2018/13481-0, São Paulo Research Foundation (FAPESP) and CNPq grants 315925/2021-3 and 434599/2018-2. DDN is partially supported by S\~{a}o Paulo Research Foundation (FAPESP) grants 2022/09633-5, 2019/10269-3, and 2018/13481-0, and by Conselho Nacional de Desenvolvimento Cient\'{i}fico e Tecnol\'{o}gico (CNPq) grant 309110/2021-1.

Data sharing not applicable to this article as no datasets were generated or analysed during the current study.

Conflict of Interest: The authors declare that they have no conflict of interest.

\bibliographystyle{abbrv}
\bibliography{CasiMartNova2024}
\end{document}